\documentclass[letterpaper,10pt]{article}

\usepackage{enumerate}
\usepackage{amsmath}
\usepackage[english]{babel}
\usepackage{amssymb}
\usepackage{epsfig}
\usepackage{pstricks,pst-node,pst-coil,pst-plot,pst-text}
\usepackage{ifpdf}
\usepackage{subfigure}
\usepackage{caption}
\usepackage{amsthm}
 \usepackage{wrapfig}
\usepackage{sidecap}
\usepackage{textcomp}

\usepackage[title,titletoc]{appendix}

\newcommand{\ccr}{\textit{\textbf{ccr}}}
\newcommand{\dcdc}{\textit{\textbf{dcdc}}}

\newtheorem{theorem}{Theorem}

\newtheorem{definition}{Definition}

\newtheorem{lemma}[theorem]{Lemma}
\newtheorem{proposition}[theorem]{Proposition}

\newtheorem{observation}{Observation}

\newtheorem{assumption}{Assumption}
\newtheorem{conjecture}{Conjecture}
\newtheorem{property}{Property}
\newtheorem{claim}{Claim}

\newtheorem*{remark}{Remark}

\newlength{\saveparindent}
\usepackage{graphicx,color}
\title{Directed cycle double covers and cut-obstacles}

\author{
  Andrea~Jim\'enez\thanks{Instituto de Matem\'atica e Estat\'istica, Universidade de S\~ao Paulo,
\texttt{ajimenez@ime.usp.br}. Supported by 
CNPq (Proc.~477203/2012-4 and~456792/2014-7), FAPESP (Proc.~2011/19978-5,~2013/23331-2 and~2013/03447-6) and Project USP MaCLinC/NUMEC.}\\
\\
  \and
  Martin~Loebl\thanks{Department of Applied Mathematics,
  Charles University, \texttt{loebl@kam.mff.cuni.cz}. 
  Partially supported by the Czech Science Foundation under the contract number P202-13-21988S.}\\
}

\date{}

\begin{document}

\maketitle
\begin{abstract}
A directed cycle double cover of a graph $G$ is a family of cycles of $G$,
each provided with an orientation, such that every edge of $G$ is covered by 
exactly two oppositely directed cycles.
Explicit obstructions to the existence of a directed cycle double cover in a graph
are bridges. Jaeger~\cite{Jaeger19851} conjectured that bridges are actually the only obstructions.
One of the difficulties in proving the Jaeger's conjecture lies in
discovering and avoiding obstructions to partial strategies that, if successful, create directed cycle double covers. 
In this work, we suggest a way to circumvent this difficulty.
We formulate a conjecture on graph connections (see~\cite{frank} for recent monograph), whose
validity follows by the successful avoidance of one cut-type obstruction that we call cut-obstacles.
The main result of this work claims that our \emph{cut-obstacles avoidance conjecture}
already implies Jaeger's directed cycle double cover conjecture.
\end{abstract}

\section{Introduction, outline of methods and main results.}
\label{sec.int}
Jaeger's directed cycle double cover conjecture~\cite{Jaeger19851} asserts that 
for every 2-connected graph $G$ there exists a family of cycles $\mathcal{C}$ of $G$
such that it is possible to prescribe an orientation to each cycle of $\mathcal{C}$
in such a way that each edge $e$ of $G$ belongs to exactly two cycles of $\mathcal{C}$ and 
these cycles induce opposite orientations on~$e$. 
Jaeger's conjecture trivially 
holds in the class of cubic bridgeless planar graphs.
And it has certainly been positively settled for some more classes of graphs,
as for example, graphs that admit a nowhere-zero
4-flow~\cite{MR1395462}
and 2-connected projective-planar graphs~\cite{Ellingham:2011:OEO:1953656.1954213}.
We kindly invite the reader interested in more details about the development of the directed
cycle double cover conjecture and related problems to consult~\cite{Jaeger19851,zhang1997integer,zhang2012circuit}.

%

A \emph{reduction network} is a communication network consisting of interconnected parts where each part has to perform some task.  
The elements of each part need to communicate in order to successfully perform the task. 
The parts are linearly ordered, and after a part performs its task, the part is no more directly functional and it is \emph{reduced} --- that is, the communication network is updated in such a way that the reduced part is removed and only its residue remains. 
The goal is to make {\em reductions} whose residue help communication in yet functional parts.

The reduction network is modeled by an undirected graph $G$, 
and its linearly ordered parts by the ears of an ear-decomposition of $G$ --- the order of the parts is reversed order of the ears.  
At the initial step, the first part is \emph{reduced} and the updated communication network becomes
a \emph{mixed graph}; that is, the union of a subgraph of $G$ and the residues of the the initial reduction.
At each step, the currently active part 
is \emph{reduced} if its \emph{correct reduction} exists. 
There are multiple obstructions to the existence of \emph{correct reductions}, 
one of them are the \emph{cut-obstacles}.
In this work, we investigate reduction networks that are modeled by {\em robust trigraphs}.

\smallskip

In what follows, we formalize this discussion.
In order to make this paper as self-contained as possible, in the next paragraphs
we repeat some definitions of~\cite{ourwork}. 

\subsubsection*{Robust trigraphs}

In this work, an \emph{ear} is a path on at least 3 vertices or a star on 4 vertices; in particular, an edge is not an ear. 
A \emph{trigraph} is a cubic graph that can be obtained 
from a cycle by sequentially adding \emph{short ears}, that is, ears on 
at most~5 vertices. 
Whenever $H$ is a trigraph, an ear decomposition of $H$ means a short-ear decomposition of $H$
and the expression $(H_0, H_i, L_i)^n$ stands for such an ear decomposition;
that is, where $H_0$ is the initial cycle of the ear decomposition, $L_i$ is the $i$-th short ear, 
$H_i$ is the intermediate graph obtained from $H_0$ by adding the first $i$ short ears
and $H_n =H$. 


For each ear $L$, we denote by $I(L)$ the set of its vertices of degree at least 2 
and if $L$ is a path we say that $L$ is a $k$-ear if the cardinality of $I(L)$ is $k$.
Let $(H_0, H_i, L_i)^n$ be an ear decomposition of $H$.
The {\em descendant} of $I(L_i)$, where $L_i$ is a 3-ear,
is the maximal subgraph $D$ of $H-V(H_i)$ such that $H[I(L_i) \cup V(D)] $ is connected;
the notation $H[X]$ represents the induced subgraph of $H$ on vertex set $X \subset V(H)$. 
In other words, $D$ is the descendant of $I(L_i)$
if and only if $D$ is the maximal subgraph of $H-V(H_i)$
such that for each component $C$ of $D$ there exists an edge 
connecting $C$ and $I(L_i)$. 

\begin{definition}[Robust]
\label{def.wr}
 Let $(H_0, H_i, L_i)^n$ be an ear decomposition of a trigraph $H$.
 We say that $(H_0, H_i, L_i)^n$ is robust if for each 3-ear, say $L_i$,
 the descendant of $I(L_i)$ is composed of at most 2 connected components. 
 Moreover, if the descendant is composed of two connected components, then one of them is an isolated vertex
 adjacent to two vertices in $V(H_0)$. 
\end{definition}

\subsubsection*{Mixed graphs}
A {\em mixed graph} is a 4-tuple $(V, E, A, R)$, where $V$ is a vertex set,
$E$ is an edge set, $A$ is a set of directed edges (arcs), and $R$ is a subset of $A{\times} A$,
that is, a set of pairs of arcs. It is require that in the graph $(V,E)$, that is, the graph on vertex set $V$ and edge set $E$, 
each vertex has degree at least one 
and at most three, and that, in $(V, E, A, R)$, each vertex of degree one (resp. two) in $(V,E)$ is the tail
of exactly two arcs (resp. one arc) and the head of exactly two arcs (resp. one arc); note that 
directed loops are allowed. 
Throughout this paper, $\{u,v\}$ denotes the (non-directed) edge with end vertices $u$, $v$ and $(u,v)$ denotes 
the arc directed from $u$ to~$v$; if no end vertices are specified $\vec{e}$ denotes an arc.

\subsubsection*{Correct reductions}

Let $(V, E, A, R)$ be a mixed graph and $U \subseteq V$. 
A {\em correct reduction} of $U$ on $(V, E, A, R)$ is a procedure that outputs 
a new mixed graph $(V', E', A', R')$ and a list $\mathcal{S}$ of directed paths and cycles (here, 
a directed loop is considered
a directed cycle)
so that the following items hold:
\begin{itemize}
 \item[i)] $V' = V-U$.
 \item[ii)] $E'=\{ \{u,v\} \in E: \{u, v\} \cap U =\emptyset\}$. 
  \item[iii)] Let  $\tilde{A}$ denote the set of arcs obtained by replacing each edge in $E$ 
        incident to a vertex of $U$ by two arcs oppositely directed and
        $A(U)$ be the subset of $A$ that contains all directed 
        edges incident to a vertex of $U$. The list $\mathcal{S}$ is the result of 
        partitioning $A(U){\cup} \tilde{A}$ into correct directed paths 
        and correct directed cycles. A cycle or a path  is {\em correct} 
        if no pair of its arcs is an element of $R$, and if it is not a 2-cycle composed of
        only arcs from $\tilde{A}$;
        in particular, a directed loop is a correct cycle. 
        In addition, a correct path must have both end vertices in $V{-}U$.

        \item[iv)] $A' = \{(x,y) \in A: \{x, y\} \cap U =\emptyset\} \cup A''$. 

The set $A''$ is the set of arcs obtained by replacing each correct directed path $P \in \mathcal{S}$ by a
        new directed edge, say $\vec{e}_P$, with both end vertices in $V{-} U$ and such that 
        $\vec{e}_P$ has the orientation of $P$. 
        
\item[v)] $R' = \{ \{(x,y), (x',y')\} \in R :  \{x, y, x', y'\} \cap U =\emptyset\} \cup R'' \cup \tilde{R}$.

The set $R''$ corresponds to all pairs $\{\vec{e}_P, \vec{e}_{P'}\}$ such that
  $P$ and $P'$ have a vertex in common, or there are arcs $\vec{e} \in P$ and $\vec{g} \in P'$ such that
  $\{\vec{e},\vec{g}\} \in R$.
  The set $\tilde{R}$ corresponds to all pairs $\{\vec{e}_P, \vec{g}\}$ such that
  $P$ contains an arc $\vec{e}$ and $\{\vec{e},\vec{g}\} \in R$.
\end{itemize}

\noindent
An example of a correct reduction is provided in Figure~\ref{fig:correct-red}.
In consequence with the definition of correct reductions, we refer to the elements in $\mathcal{S}$ as \emph{correct paths}
and \emph{correct cycles},
 and to the elements in $R$ as \emph{forbidden pairs}.

\begin{figure}[h]
\centering
\subfigure[]
{
\ifpdf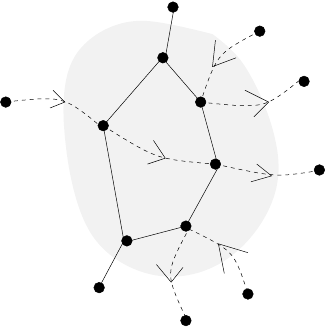\else\input{correct_reduction.ps_t}\fi
\label{fig.srmg00}
}
\subfigure[]
{
\ifpdf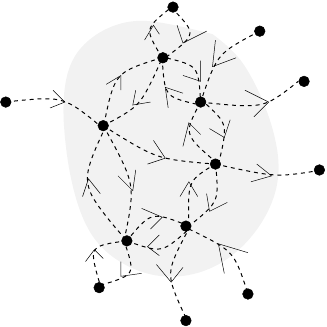\else\input{correct_reduction_1.ps_t}\fi
 \label{fig.srmg01}
}
\subfigure[]
{
\ifpdf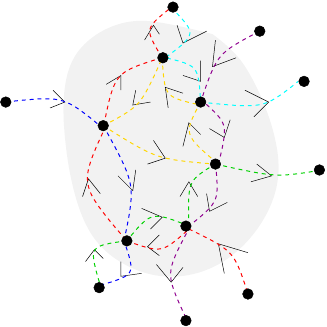\else\input{correct_reduction_2.ps_t}\fi
\label{fig.srmg02}
}
\subfigure[]
{
\ifpdf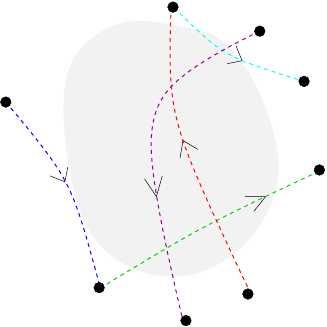\else\input{correct_reduction_3.ps_t}\fi
 \label{fig.bf3}
}
\caption{Correct reduction of $U \subset V$:
(a) elements from $A$ are re\-presented by dotted lines
and the pair $(\vec{e}, \vec{g})$ of arcs is not in $R$,
(b) replace edges of $E$ with an end vertex in $U$ by two arcs oppositely directed,
(c) partition of $A(U) \cup \tilde{A}$ into correct paths and cycles and
(d) resulting structure.}
\label{fig:correct-red}
\end{figure}

 Let $G$ be a cubic graph and $V_0, V_1,\ldots, V_k$ be a partition~of $V(G)$ such that for all $i\in[k]$, the graph $G-(V_k \cup \cdots \cup V_i)$ is connected; by abuse of notation, we refer to such a partition as \emph{connected}.
Let $j \in \{0,1,\ldots,k\}$. A \emph{consecutive correct reduction} (\textit{\textbf{ccr}}, in short) of $V_k, V_{k-1},\ldots, V_j$ is a sequence of $k-j+1$ correct reductions such that 
for each $i \in \{k, \ldots, j\}$, the correct reduction of $V_{i} \subset W_{i}$ on $(W_{i} ,E_{i}, A_{i}, R_{i})$ 
outputs $(W_{i-1} ,E_{i-1}, A_{i-1}, R_{i-1})$; 
where $(W_{k} ,E_{k}, A_{k}, R_{k}) = (V(G) ,E(G), \emptyset, \emptyset)$ 
and  $(W_{-1} ,E_{-1}, A_{-1}, R_{-1})= (\emptyset, \emptyset, \emptyset, \emptyset)$.

The following proposition is proved in Section~2 of~\cite{ourwork}.
It relates the notion of consecutive correct reductions to the one of a directed cycle double cover ({\dcdc}, in short).

\begin{proposition}\label{prop:consecorr}
Let $V_0, V_1,\ldots, V_k$ be a connected partition~of the vertex set of a cubic graph $G$.
Then each {\ccr} of $V_k, V_{k-1},\ldots, V_0$ constructs a {\dcdc} of~$G$.
\end{proposition}

\subsubsection*{Cut-obstacles}

There are many potential obstructions to the existence of correct reductions. One of them is a cut-obstacle.

\begin{definition}[Cut-obstacle]
\label{def.cuto}
Let $(V, E, A, R)$ be a mixed graph and  $U\subseteq V$.
We denote by $C_U$ the subset of $E \cup A$ that contains all edges and arcs with exactly one end vertex in $U$. 
We say that $U$ is a {\em cut-obstacle} in $(V, E, A, R)$ if  the number of edges in~$C_U$ is strictly less than twice the number of arcs in $C_U$.
\end{definition}
For an illustration of a cut-obstacle see Figure~\ref{fig.cut3ear}. 

Cut-obstacles are potential obstacles for the existence of correct reductions in the following sense: if we 
assume that all pairs of arcs in $C_U$ belong to $R$, then there is not correct reduction of $U$ on~$(V, E, A, R)$.

The next observation helps understanding the connection between 
robust ear decompositions and cut-obstacles. 
\begin{observation}\label{obs:cuto1}
If $(H_0, H_i, L_i)^n$ is an ear decomposition of a trigraph $H$ and there exists
a 3-ear $L \in \{L_1, \ldots, L_n\}$ such that its descendant is composed of 3 connected
components, then any {\ccr} of $I(L_n), I(L_{n-1}),\ldots, I(L_1)$ creates a cut-obstacle at $I(L)$.
\end{observation}

In other words, Observation~\ref{obs:cuto1} says that, for a 3-ear, 
having at most 2 connected components in its descendant is fundamental in order to not creating
trivial cut-obstacles. Therefore, in a robust ear decomposition (Definition~\ref{def.wr}), there are no trivial cut-obstacles
at any 3-ear. The moreover part of Definition~\ref{def.wr} is imposed by a similar technical reasons:
roughly, this condition avoids trivial cut-obstacles at 3-ears 
that belongs to the ear decompositions of the trigraphs that model cubic bridgeless graphs.

\bigskip

In the following, we define ears decompositions that admit 
{\ccr} that does not create cut-obstacles. In the subsequent observation we illustrate it for the planar trigraphs.

\begin{definition}[Superb]\label{def.superb}
Let $H$ be a trigraph. We say that an ear decomposition $(H_0, H_i, L_i)^n$ of $H$
is \emph{superb} if there exists a {\ccr} of $I(L_n), I(L_{n-1}), \ldots, I(L_{j})$ 
for some $j \in [n]$ such that the following properties hold:
\begin{itemize}
 \item[(i)]  $j=1$ or $I(L_{j-1})$ cannot be correctly reduced, and
 \item[(ii)] a cut-obstacle at the internal vertices of a 3-ear in $\{L_{j-1}, L_{j}, \ldots, L_{n}\}$ is never created.
\end{itemize}
\end{definition}

\begin{observation}
\label{o.av}
Every robust ear decomposition of a planar trigraph $H$ is superb.
\end{observation}
\begin{proof}
Let us consider an embedding of $H$ in the plane and let $\mathcal{C}$ be the set of facial cycles defined
by this embedding. The set $\mathcal{C}$ is a {\dcdc} of $H$: we can prescribe clockwise orientation
to the facial cycles of interior faces and anticlockwise orientation to the facial cycle of the external face.
Let $(H_0, H_i, L_i)^n$ be a robust ear decomposition of $H$. 
Observe that $\mathcal{C}$ encodes a {\ccr}  of $I(L_n), \ldots, I(L_{1}), V(H_0)$. 
We claim that such {\ccr} does not create cut-obstacles at any 3-ear.
The rest of the proof is devoted to prove this.

Let $L_i$ be a 3-ear, $i \in [n]$ and $V(L_i)=\alpha v^1 v^2 v^3 \beta$.
We assume that $I(L_{i+1}), \ldots, I(L_{n})$ has been already correctly reduced according to $\mathcal{C}$
and let $H'$ denote the obtained mixed graph.

Suppose first that there are adjacent internal vertices of $L_i$, without loss of generality $v^1$ and $v^2$,
such that if $x$ (resp. $y$)
is the neighbour of $v^1$ (resp. of $v^2$) that is not in $V(L_i)$, then 
$x$ and $y$ are in the same component, say $D$, of the descendant of $I(L_i)$.
By planarity and since $H$ is cubic, it holds that the path $xv^1v^2y$ is a subgraph of a facial cycle $C$ of $\mathcal{C}$.
Again by planarity and using the fact that $V(D)$ is connected, we have that all vertices in $V(C)\setminus\{v^1, v^2\}$ belong to $V(D)$.
If so, all vertices in $V(C)\setminus\{v^1, v^2\}$ are internal vertices
of ears that were already reduced and thus, in $H'$, there is an arc with end vertices $v^1, v^2$.

We now suppose that there are no such internal vertices. Since $(H_0, H_i, L_i)^n$ is weakly robust,
it implies that $x, z$, where $z$ is the neighbour of $v^3$ that is not in $V(L_i)$ are in the same 
component, say $D'$, of the descendant of $I(L_i)$. Moreover, the component of the descendant of $I(L_i)$ to which
$y$ belongs contains only $y$  (thus, $y$ in an internal vertex of a star in $L_{i+1}, \ldots, L_{n}$)
and in the planar embedding, $y$ is not in the same side (with respect to $L_i$) where $x$ and $z$ are
since $(H_0, H_i, L_i)^n$ is weakly robust and thus $H$ is 2-edge-connected.
As before, it follows that the path $xv^1v^2v^3z$ belongs to a path that is a subgraph of a facial cycle $C'$ of $\mathcal{C}$
and all vertices in $V(C')\setminus\{v^1, v^2, v^3\}$ are in $V(D)$.
We conclude that in $H'$ there is an arc with end vertices $v^1, v^3$.
\end{proof}

\begin{remark}
In Conjecture~\ref{conj:weakavoidance} we propose that a weakening of Observation~\ref{o.av} holds for every 
robust trigraph. The weakening is twofold: we allow changing the ear decomposition and also the trigraph itself. 
This is explained next. 
\end{remark}

\begin{definition}[Admits $H_0, S$]
\label{def.S}
Let $H$ be a trigraph, $H_0$ be an induced cycle of $H$ and let  $(H_0, H_i, L_i)^n$
be an ear decomposition of $H$. 
\begin{itemize}
\item
We refer to a 2,3-ear as a \emph {base} if its leaves are in $H_0$, as an \emph {up} if one leaf belongs to $H_0$ and the other one to a base, as an \emph{antenna} if exactly one leaf belongs to an up.
\item
Let $S$ be a set of disjoint paths of $H$ on 3 vertices. We say that $(H_0, H_i, L_i)^n$ 
\emph{admits $H_0, S$} if the following three conditions hold.
(1) If $L_i$ is a 3-ear, then $L_i$ is either a base, or an up or, an antenna containing an element of $S$.
(2)~Each element of $S$ is a subset of an antenna.
(3)~No 1-ear has both leaves in $H_0$ or one leaf in a base and second one in an up.
\end{itemize}
\end{definition}

The following notation helps with the next definition. Given an ear decomposition $(H_0, H_i, L_i)^n$
of a trigraph, we say that a sequence $(L_{i_j})_{j \in [l]}$ of ears is a \emph{heel}
if $L_{i_1}$ is a 2-ear, for every $j \in \{2,\ldots,l\}$
the leaves of $L_{i_j}$ are exactly the internal vertices of $L_{i_{j-1}}$
and $(L_{i_j})_{j \in [l]}$ is maximal.
Note that $L_{i_{j}}$ is either a 2-ear, or a 1-ear.

\begin{definition}[Local exchange. Closure]
\label{def.closure}
Let $H$ be a trigraph, $H_0$ be an induced cycle of $H$ and let $S$ be a set of disjoint paths of $H$ on 3 vertices. 
Let $\mathcal{H}=(H_0, H_i, L_i)^n$  be an ear decomposition of $H$ that admits $H_0, S$. 
\begin{itemize}
\item A {\em local exchange} on the pair $H, \mathcal{H}$
is an operation that produces a pair $H', (H_0, H'_i, L'_i)^n$ of a trigraph with its ear decomposition (see Figure~\ref{fig:localexchange}) as follows:
Let $L_{i_0}\in \{L_1,\ldots, L_n\}$ be a 3-ear antenna with vertices $a, w_1, w_2, w_3, b$. 
Let $(L_{i_j})_{j \in [l]}$ be a heel such that the leaves of $L_{i_1}$ are $w_2, w_3$,
and let $u$ be the internal vertex of $L_{i_l}$ that has the shortest distance (in the heel) to $w_3$. 
Further let $L_k, L_m$ be ears such that $w_1$ is a leaf of $L_k$ and $u$ is a leaf of $L_m$
and assume $\{w_1,w'_1\}$, $\{u,u'\}$ are edges of $L_k$, $L_m$, respectively. 
Then, $H', L'_k, L'_m$ are obtained by deleting the edges $\{w_1,w'_1\}$, $\{u,u'\}$ 
and adding the new edges $\{w_1,u'\}$, $\{u,w'_1\}$. The other ears do not change, namely $L'_i=L_i$ for all $i \notin \{k,m\}$.
Note that $(H_0, H'_i, L'_i)^n$ admits $H_0, S$.
\item
The \emph{closure} of the pair $H$, $\mathcal{H}$ is the set of all pairs $H'$,
$\mathcal{H}'$, where $H'$ is a trigraph and $\mathcal{H}'$ is an ear decomposition of $H'$ admitting $H_0, S$,
such that $H'$, $\mathcal{H}'$ are  obtained from $H$, $\mathcal{H}$ by a sequence of the following two operations: 
(1) a modification of the current ear decomposition to another one admitting $H_0, S$,  and (2) a local exchange on the current trigraph.
\end{itemize}
\end{definition}

\begin{figure}[h]
 \centering
 \ifpdf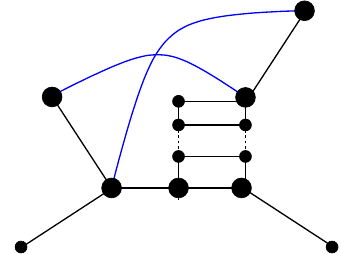\else\input{localexchange.ps_t}\fi
 \caption{Illustration of a local exchange: edges $\{w_1,w'_1\}$, $\{u,u'\}$ are deleted, and new edges $\{w_1,u'\}$, $\{u,w'_1\}$ (the ones in blue)
 are added.
 The case depicted corresponds to the one that the last ear of the heel, namely $L_{i_l}$, is a 2-ear.}
 \label{fig:localexchange}
\end{figure}

Note that in definition of closure (Definition~\ref{def.closure}), a modification changes only ear decomposition, while 
in a local exchange both, the ear decomposition and the trigraph change.

\subsection{Main Contribution}
\label{sub.mainresults}

Aside of cut-obstacles, there are many other obstacles for the existence of a correct reduction. The main contribution of this paper is a proposition which claims that \emph {avoidance of the cut-obstacles in a very restricted setting} is sufficient for proving the {\dcdc} conjecture.

\begin{conjecture}[Cut avoidance conjecture] 
\label{conj:weakavoidance}
Let $H$ be a trigraph, $H_0$ be an induced cycle of $H$ and let $S$ be a set of disjoint paths on 3 vertices. 
Let $\mathcal{H}$ be a robust ear decomposition of $H$ that admits $H_0, S$.
Then, there is a pair $H'$, $\mathcal{H}'$ in the closure of $H$, $\mathcal{H}$,
such that $\mathcal{H}'$  is superb.
\end{conjecture}

The main result of this paper is the following.

\begin{theorem}\label{th:mainmain}
If Conjecture~\ref{conj:weakavoidance} holds, then 
the {\dcdc} conjecture holds in general graphs.
\end{theorem}

The rest of the paper is organized as follows.
In Section~\ref{sec:main} we describe the essential step in the proof
of Theorem~\ref{th:mainmain}. 
Firstly, for each cubic graph~$G$ and its ear decomposition, 
we construct a trigraph $H(G)$ with specified induced cycle $H_0= H_0(G)$ and set $S= S(G)$. 
Simultaneously, we construct ear decomposition $(H_0, H_i, L_i)^n$ of $H(G)$ that admits $H_0, S$. 
The ear decomposition $(H_0, H_i, L_i)^n$ will be robust if the ear decomposition of $G$ used in the construction
of  $H(G)$ is {\em super robust}. 
In Section~\ref{sec:3conn} of this paper we define super robust ear decompositions and
prove that every 3-edge-connected cubic graph admits a super robust ear decomposition. 
We recall that the {\dcdc} conjecture is as hard for 3-edge-connected cubic graphs as for general graphs.

In Section~\ref{sec:main1}, we introduce the concept of \emph{relevant} ear decomposition
and show that superb relevant ear decompositions encode 
directed cycle double covers of~$G$. 
Then (in Lemma~\ref{l.last}) we extend this claim to each ear decomposition in the closure of $(H_0, H_i, L_i)^n$,
thus proving Theorem~\ref{th:mainmain}.

\begin{remark} In this remark we explain the reason behind the definition of the {local exchange} operation.
Assume that 3-edge-connected cubic graph $G$ has a {\dcdc}. 
Does then the constructed trigraph $H(G)$ and its constructed ear-decomposition $\mathcal{H}$ satisfy Conjecture~\ref{conj:weakavoidance}? 
A {\dcdc} of $G$ defines an embedding of $G$ in an orientable surface with no dual loop.
Such embedding gives rise to a special ear-decomposition, as in the toroidal example of Section~\ref{sec.surfaces}.
Let $H', \mathcal{H}'$ be the trigraph and its ear decomposition constructed from this special ear-decomposition. 
Possibly $H$ is not isomorphic to $H'$, but we believe that  $H', \mathcal{H}'$  is in the closure of  $H, \mathcal{H}$. 
We further conjecture that $\mathcal{H}'$  is superb, which is illustrated by the toroidal example of Section~\ref{sec.surfaces}.
\end{remark}

\section{Trigraph $H(G)$}
\label{sec:main}

Let $G$ be a cubic graph and $(G_0, G_i, P_i)^{k}$ be an 
ear decomposition of $G$; we recall that, an ear is a path on at least three vertices or a star on four vertices.  
The aim of this section is to construct the trigraph $H(G)$ along with an ear decomposition. 

Let $v_0$ be a fixed vertex of $G_0$. Let $v_1$ and $v_2$ denote the neighbours of $v_0$ in $G_0$.
We obtain a cubic graph $G'$ from $G$ by subdividing edges $\{v_0,v_1\}$ and $\{v_0,v_2\}$ into
$\{v_0,x_0\}$, $\{x_0,v_1\}$ and $\{v_0,y_0\}$, $\{y_0,v_2\}$, respectively, and adding the new edge $\{x_0,y_0\}$; this operation is 
known as a Y-$\Delta$ operation.
The cubic graph~$G'$ admits an ear decomposition starting at the triangle on vertex set 
$\{x_0,y_0,v_0\}$, and with building ears $P_0, P_1, \ldots, P_k$, where 
  $P_0$ is the path obtained from $G_0$ by deleting $v_0$ and 
  adding the edges $\{x_0,v_1\}$, $\{y_0,v_2\}$.
Clearly, $G$ has a {\dcdc} if and only if $G'$ does so. 
In the rest of this section, for each  $i \in \{0,1,\ldots,k\}$, 
the notation $a_i$, $c_i$ stand for the end vertices of $P_i$, whenever $P_i$ is a path; in particular
$\{a_i, c_i\}=\{x_0, y_0\}$ is the set of end vertices of $P_0$.

Let $H(G)$ be a cycle on $n(G)$ vertices, with $n(G)$ as described in Remark~\ref{o.cl1}.
Set $H_0$ as the starting cycle of the ear decomposition of~$H(G)$ and
choose 3 distinct vertices from $V(H_0)$; we denote the set of these vertices by $V_0$.

The following building block comes in handy to describe the construction~of~$H(G)$.

\begin{definition}[Basic gadget]\label{def:basicgadget}
A \emph{basic gadget} $\mathcal{B}=\mathcal{B}(x,u)$ is a sequence $E_1, E_2, E'_3, E_4, D_1, D_2, D_3$ of short
ears which, considering the vertices named according to Figure~\ref{fig:bagadget}, are defined as follows:  
\begin{itemize}
 \item[(i)] $I(E_1)=\{a',w',b'\}$, $I(E_2) =\{a,w,b\}$, $V(E'_3)=\{b,z,y,x\}$, $V(E_4)=\{z, u, v, y\}$,
 and the end vertices of $E_1$ and $E_2$ belong to $H_0$,
 \item[(iii)] $D_1, D_2, D_3$ are stars, $\{a', w, v\}$ is the set of leaves of $D_1$, and each star
 $D_2$, $D_3$ has two leaves in $H_0$.
 Moreover, vertex $w'$ is a leaf of $D_2$ and vertex $a$ is a leaf of $D_3$.
\end{itemize}
Depending on the context, a basic gadget may also refer to the 
graph obtained by the union of the ears $E_1, E_2, E'_3, E_4, D_1, D_2, D_3$. 
In addition, we say that the path on vertex set $\{b,z,y\}$ is the {\emph{fixed}} path,
$u$ is the \emph{replica} and $x$ is the \emph{joint} of the basic gadget; 
whenever only the replica vertex $u$ is specified, we denote by $x_u$ the corresponding joint vertex.
\end{definition}

\begin{figure}[h] 
\centering 
\subfigure[Basic gadget $\mathcal B(x,u)$]
{
\ifpdf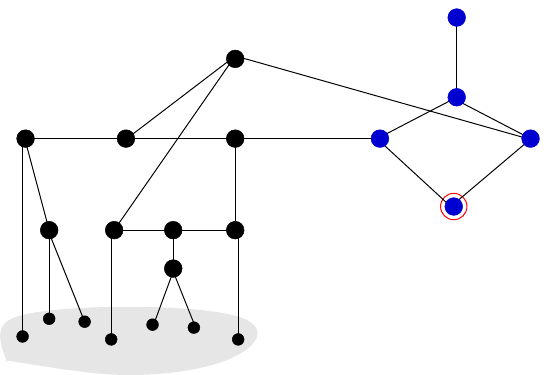\else\input{gadget.ps_t}\fi
\label{fig:bagadget}
}\qquad
\subfigure[Gadget after reduction of its stars.]
{
\ifpdf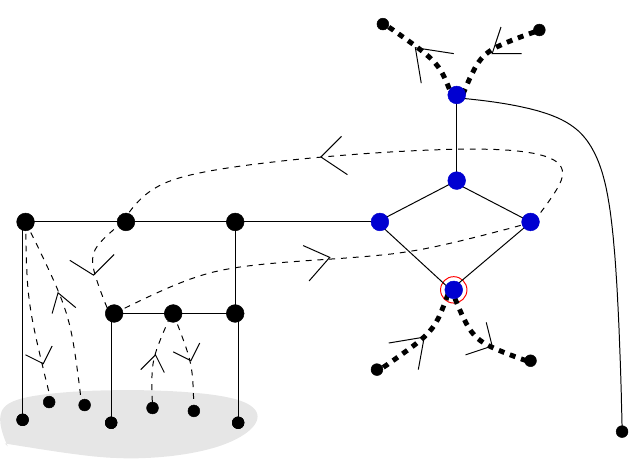\else\input{gadgetgadget.ps_t}\fi
 \label{fig:gadget}}
 \caption{}
 \label{fig:basicandno}
\end{figure}

We now describe the recursive construction of $H(G)$, 
along with a function~$\Gamma$ that maps the set of the vertices and edges of $G'$ into a subset of the vertices and edges of $H(G)$. 
In addition, we define a set $S$ of disjoint paths of $H(G)$ on 3 vertices, 
and an ear decomposition of $H(G)$ that we call {\em canonical}. 

Recall that the ear decomposition of $H(G)$ starts at cycle $H_0$.
Set $V_0 = \{\Gamma(x_0), \Gamma(y_0),  \Gamma(v_0)\}$ and $S_0 = \emptyset$.
According to the following rules, for each $i \in \{0,1,\ldots,k\}$, 
we obtain $H_{t_i}$ from $H_{t_{i-1}}$,
$S_{t_i}$ from $S_{t_{i-1}}$, and the list $\mathcal{L}_{t_{i}}$ 
of short ears of the canonical ear decomposition of $H(G)$ 
that generates $H_{t_i}$ from $H_{t_{i-1}}$;
under this notation, $H_{t_{-1}}=H_0$ and $S=S_{t_k}$. 
\begin{itemize}
    \item Case that $P_i$ is a star with three leaves. Let $s$ denote the center of $P_i$ and let $x, y, z$ denote its leaves.
     Graph $H_{t_i}$ is obtained from $H_{t_{i-1}}$ by adding the new vertex $\Gamma(s)= u$,
     and the new edges $\{u, \Gamma(x)\}$, $\{u, \Gamma(y)\}$ and $\{u, \Gamma(y)\}$. 
     For $w\in \{x,y,z\}$ let $\Gamma(\{s,w\})= \{\Gamma(s), \Gamma(w)\}$. 
     Further, $S_{t_i}=S_{t_{i-1}}$, and
     $\mathcal{L}_{t_{i}}$ contains only one element:      
     the star on vertex set $\{u, \Gamma(x), \Gamma(y), \Gamma(y)\}$.
  
\item Case that $P_i$ is a path. Let  $a_i, b_{i_1}, \ldots, b_{i_l}, c_i$, with $l\geq 1$, be the sequence of vertices of path~$P_i$. 
In order to obtain $H_{t_i}$ from $H_{t_{i-1}}$, we first 
add {\em substantial path} $Q_i$, which is defined by the sequence of vertices $\Gamma(a_i), x_1, \ldots, x_l, \Gamma(c_i)$, where 
$x_1, \ldots, x_l$ are $l$ new vertices. 
Then to each $x_i$, we connect a basic gadget graph $\mathcal B_i=\mathcal B(x_i,u_i)$; 
consequently, each of them is referred to as a \emph{basic gadget of $H(G)$}.
Set $\Gamma(b_{i_j})= u_j$, for each $j\in \{1, \ldots, l\}$. 
Let~$\phi$ denote the natural isomorphism between $P_i$ and $Q_i$
that maps $a_i$ to $\Gamma(a_i)$ and $c_i$ to~$\Gamma(c_i)$.
For each edge $\{u,v\} \in E(P_i)$, let $\Gamma(\{u,v\})= \{\phi(u),\phi(v) \}$.
Note that, under this setting, $\Gamma(\{u,v\})\neq \{\Gamma(u), \Gamma(v)\}$.  
In addition, $S_{t_i}$ is the union of the paths in $S_{t_{i-1}}$
and the $l$ fixed paths of the basic gadgets $\mathcal B_i, 1\leq i\leq l$ (see Definition~\ref{def:basicgadget}). 
Finally, $\mathcal{L}_{t_{i}} = (R_1, \ldots, R_l)$, 
where for each $1\leq j< l$, $R_j$ 
is the list of ears 
\begin{equation} \label{eq:gadget}
E_1(\mathcal B_j),E_2(\mathcal B_j),E_3,E_4(\mathcal B_j),D_1(\mathcal B_j),D_2(\mathcal B_j),D_3(\mathcal B_j), 
\end{equation}
 $\mathcal{B}(x_j,u_j)\cup\{x_j,x_{j-1}\}$, 
where for $j=1$ we let $x_0= \Gamma(a_i)$, and $E_3$ is the path obtained by the union of $E'_3(\mathcal B_j)$ and $\{x_j,x_{j-1}\}$.
The last list $R_l$ consists of the ears \begin{equation} \label{eq:gadget*} E_1(\mathcal B _l),E_2(\mathcal B _l),F,E'_3(\mathcal B _l),E_4(\mathcal B _l),D_1(\mathcal B_l),D_2(\mathcal B_l),D_3(\mathcal B_l),\end{equation}                                                                                                                                                                                                                                                                                                                                                                                                                               
where $F$ is the path on vertex set $\{x_{l-1},x_l, \Gamma(c_i)\}$.
\end{itemize}

Note that $S$ is exactly the set of all fixed paths of the basic gadgets of $H(G)$.
It is a routine to check that the canonical ear decomposition of $H(G)$ admits $H_0,S$. 
Further, if $\mathcal{B}$ is a basic gadget of $H$, then
$E_1(\mathcal B)$ is a base, $E_2(\mathcal B)$ is an up
and, $E_3$ in (\ref{eq:gadget}), $E_3'$ in (\ref{eq:gadget*}) are antennas (see Definition~\ref{def.S}).
The following remark sets $n(G)$; number of vertices of $H_0$.

\begin{remark}\label{o.cl1}
Each path $P_i$ in $\{P_1,\ldots,P_k\}$ gives rise to $|I(P_i)|$ distinct basic gadgets in $H(G)$.
Moreover, in order to construct each basic gadget in $H(G)$, 7 vertices from $H_0$ are needed.
We set,
$$n(G) = 7 \left(\sum_{i \in [k] \,:\, P_i \, \text{path}} |I(P_i)|\right) +3.$$
\end{remark}

Because of Lemmas~\ref{lemma:supimplrob} and~\ref{th:robust} (see Section~\ref{sec:3conn}), 
the canonical ear decomposition is robust provided that graph $G$ is 3-edge-connected
and the chosen ear decomposition $(G_0, G_i, P_i)^{k}$ is  super robust. 

Next observation follows directly from the construction of $H(G)$
and the definition of closure.
Recall that, the cubic graph $G'$ is obtained from $G$ by a $Y-\Delta$ operation, as described
in the beginning of Section~\ref{sec:main}.

\begin{observation}
\label{o.cl}
 Let $\mathcal H$ be the canonical ear decomposition of $H=H(G)$ 
 and let $H', \mathcal{H}'$ be a pair in the closure of $H$, $\mathcal H$.
 By definition, $\mathcal{H}'$ admits $H_0,S$.
 Therefore, each basic gadget of $H$ is a subgraph of $H'$,
 and $G'$ is obtained from $H'$ by contracting the set of all vertices of each basic gadget of $H$ to a single vertex.
\end{observation}

Observation~\ref{o.cl} implies that the definition of function $\Gamma$ can be extended 
from a canonical ear decomposition to every ear decomposition that belongs to its closure.
This is formalized in the next definition. Note that Definition~\ref{def.ggg} is consistent with the definition of function $\Gamma$ for canonical ear decompositions.

\begin{definition}
\label{def.ggg}
Let $\mathcal H$ be the canonical ear decomposition of $H=H(G)$ 
 and let $H', \mathcal{H}'$ a pair in the closure of $H$, $\mathcal H$.
 We define function $\Gamma$ from the set of vertices and edges of $G'$ to the set of vertices and edges of $H'$ as follows. 
If $e\in E(G')$, then $\Gamma(e)$ is the edge of $H'$ which 
corresponds to $e$ after the vertex-contraction of all the basic gadgets of $H$ in $H'$. 
Let $v\in V(G')$ and let $W$ be the subset of vertices of $H'$ such that contraction of $W$ to a single vertex corresponds to $v$.
If $W$ is a single vertex, say $W= \{w\}$, then $\Gamma(v)= w$. 
Otherwise, $W$ is the vertex set of a basic gadget, and $\Gamma(v)= u$, 
where $u$ is the replica vertex of the basic gadget.  
\end{definition}

\section{Gadget analysis}\label{sec:main1}

In order to analyze the canonical ear decomposition, we extend definition of basic gadget to the one of {\em gadget}.
A gadget is basically the list of ears defined in~(\ref{eq:gadget}), which is a subsequence of the canonical ear decomposition,
and it is always associated to a basic gadget. Recall that basic gadgets are defined in~Definition~\ref{def:basicgadget}.
In the rest of this work, a \emph{block} of a sequence $a_1, \ldots, a_n $ is a subsequence $a_j, \ldots, a_m $
for $1 \leq j\leq m \leq n$. Moreover, the notation $(L_{i_1}, \ldots, L_{i_t})^{-1}$ stands for $I(L_{i_t}), \ldots, I(L_{i_1})$.

\begin{definition}[Gadget]\label{def:gadget}
Let $(H_0, H_i, L_i)^n$  be an ear decomposition of a trigraph
and $\mathcal{B}=\mathcal{B}(x,u)$ be a basic gadget.
A \emph{gadget} $\mathcal{G}$  (see Figure~\ref{fig:gadget}) of $(H_0, H_i, L_i)^n$ associated to $\mathcal{B}$   
is the block of ears $$\mathcal{G}=E_1(\mathcal{B}), E_2(\mathcal{B}), E_3, E_4(\mathcal{B}), D_1(\mathcal{B}), D_2(\mathcal{B}), D_3(\mathcal{B})$$ 
of $L_1, \ldots, L_n$, where  $E_3$ is the 3-ear on edge set $E'_3(\mathcal{B}) \cup \{x',x\}$ with
 $x'$ a neighbour of $x$ not in~$V(E'_3(\mathcal{B}))$.
\end{definition}
We may also say that $\mathcal{B}(x,u)$ is the basic gadget of $\mathcal{G}$.
Naturally, we use the terminology defined for basic gadgets on the gadgets as well.

\begin{definition}[Reduction process]\label{def:redprocess}
A {\em reduction process} of gadget $\mathcal{G}$ is a {\ccr} 
of a subsequence $\mathcal{G}'$ of $\mathcal{G}^{-1}$, so that
\begin{itemize}
 \item[(i)] either $\mathcal{G}' = \mathcal{G}^{-1}$, 
 \item[(ii)] or $\mathcal{G}'= I(D_3), I(D_2), I(D_1), I(E_4), \ldots, I(E_{j})$ for some $2 \leq j \leq 4$ 
such that $I(E_{j-1})$ does not have correct reduction or $I(E_{j-1})$ is a cut-obstacle. 
\end{itemize}
If (i) holds, we say that the reduction process is \emph{complete}.
Otherwise, we refer to it as \emph{$j$-incomplete}.
\end{definition}

\subsubsection*{Correct reductions in 3-ears}
We now mention a result from~\cite{ourwork}  which describes 
the behavior of 3-ears with respect to consecutive correct reductions.
This result is used in the proof of Theorem~\ref{theo:important}.

\begin{definition}[Inner obstacle]\label{def.inner}
An inner obstacle at a 3-ear with internal vertices $\{v^1, v^2, v^3\}$ in a mixed graph $(V,E,A,R)$ is the configuration depicted in Figure~\ref{fig.inner} such that $\{\vec{e},\vec{g}\} \in R$.
\begin{figure}[h]
\centering
 \subfigure[Cut-obstacle at  $\{v^1, v^2, v^3\}$]
 {
 \ifpdf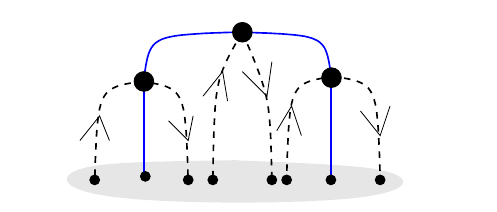\else\input{cut3ear.ps_t}\fi
 \label{fig.cut3ear}
 }\qquad
 \subfigure[Inner Obstacle]
 {
 \ifpdf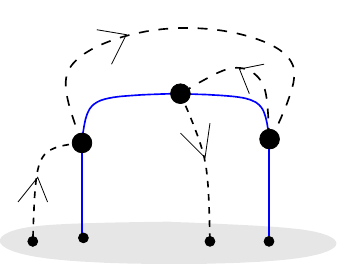\else\input{obs3ears.ps_t}\fi
  \label{fig.inner}
 }
\caption{Potential obstacles to performing correct reductions in a 3-ear.
In the case of the inner obstacle: there exist a correct reduction if and only if $\vec{e}$ and $\vec{g}$ are not forbidden
(that is, $\{\vec{e},\vec{g}\} \notin R$).}
\label{fig:obs3ears}
\end{figure}
\end{definition}

The proof of the following statement can be found in~\cite{ourwork} (Theorem 5).

\begin{theorem}
\label{thm.main2}
Let $H$ be a trigraph and $(H_0, H_i, L_i)^n$ be an ear decomposition of $H$.
Let $i \in [n]$ and $L_i$ be a 3-ear.
Let $H'_i$ denote the mixed graph obtained by a {\ccr} of  $I(L_n), \ldots, I(L_{i+1})$. 
Assume that $I(L_i)$ is not a cut-obstacle in $H'_i$.
Then, $I(L_i)$ does not admit a correct reduction if and only if $I(L_i)$ is an inner obstacle.
\end{theorem}

From now on let $(H_0, H_i, L_i)^n$ be the ear decomposition of a trigraph $H$
such that the gadget~$\mathcal{G}$, defined by the sequence $E_1, E_2, E_3, E_4, D_1, D_2, D_3$ 
(see Definition~\ref{def:gadget}), is a block of $L_1, \ldots, L_n$; 
without loss of generality, assume that $D_3=L_m$.
As in the definition of gadget, we consider the vertices of 
$\mathcal{G}$ named according to Figure~\ref{fig:gadget}. 
Moreover, we make the following assumption.

\begin{assumption}\label{assumption1}
There exists a {\ccr} of  $I(L_n), \ldots, I(L_{m+1})$ on $H$. 
\end{assumption}
\noindent
Let $H'$ denote the mixed graph obtained by a {\ccr} of $I(L_n), \ldots, I(L_{m+1})$. Let us recall that $V(H_0), I(L_1), \ldots, I(L_{m})$ is a partition
of the vertex set of $H'$. 
By the definition of gadget, we have that the vertices $x$ and $u$ have degree 2 in $H'$, 
and thus, each of them is the tail of one arc and the head of one arc.
Let us denote such arcs, according to Figure~\ref{fig:gadget}, by $\alpha_1=(x,x_1)$, $\alpha_2=(x_2, x)$
and $\beta_1=(u_1, u),  \beta_2=(u,u_2)$.
All the remaining vertices of $\mathcal{G}$ in $H'$ have degree~3;
recall that we are considering vertices named according to Figure~\ref{fig:gadget}.
Our aim is to prove the following statement.

\begin{theorem}\label{theo:important}
Each reduction process  of $\mathcal{G}$ on $H'$ satisfies exactly one of the following statements.
\begin{itemize} 
 \item[(i)] It is complete and either $\alpha_1=\beta_1$, or $\alpha_2=\beta_2$.
  \item[(ii)]  It is $j$-incomplete and $I(E_{j-1})$ is a cut-obstacle for some $2\leq j \leq 4$. 
\end{itemize}

In addition, if Statement~(i) holds and $\alpha_1 = \beta_1$ (resp. $\alpha_2=\beta_2$), then
$\alpha_2$ and $\beta_2$ (resp. $\alpha_1$ and $\beta_1$) belong to the 
two distinct correct paths (defined by the reduction process) that contain $(x,x')$ and $(x',x)$.
\end{theorem}

\begin{proof}
Let us first suppose that a complete reduction process of $\mathcal{G}$ has been performed on $H'$.
In order to prove the statement of the theorem we need to show that either $\alpha_1=\beta_1$, or $\alpha_2=\beta_2$.
For that, let us examine the single correct reductions involved in the {\ccr}
on $\mathcal{G}$ that witness the existence of the considered complete reduction process. By definition
of complete reduction process, no cut-obstacles at 3-ears are created. 
Without loss of generality, by symmetry of the stars, we can assume that the local 
configuration depicted in Figure~\ref{fig:gadget} is the one generated by the correct reduction of $I(D_3), I(D_2)$ and $I(D_1)$.
Let $\tilde{H}$ denote the obtained mixed graph. Moreover, let $e=(v,w)$ and $e'=(a',v)$ in $\tilde{H}$; 
recall again that we consider vertices named according to Figure~\ref{fig:gadget}.
We claim that the following holds. 

	  \begin{observation}\label{obs:useful}
	  In each {\ccr} of $I(E_4), I(E_3), I(E_2), I(E_1)$ on $\tilde{H}$ 
	  that creates no cut-obstacles at $I(E_3)$, $I(E_2)$ and $I(E_1)$, 
	  the arc $e$ belongs to the correct path that contains $(b,z)$
	  and $e'$ belongs to the correct path that contains $(z,b)$.
	  \end{observation}
	  \begin{proof}[Proof of Observation~\ref{obs:useful}]
	  The hypothesis that no cut-obstacles at $I(E_3)$, $I(E_2)$ and $I(E_1)$ are created, implies that the 
	  {\ccr} of 
	  $I(E_4), I(E_3)$ generates at least one arc with both end vertices in $I(E_2)=\{a,w,b\}$,
	  otherwise we have that $I(E_2)$ is a cut-obstacle. In order to get such an arc, 
	  the {\ccr} of $I(E_4), I(E_3)$ is so that $e$ belongs to the correct path that contains $(b,z)$.
	  Now, for the sake of contradiction, we assume that in a {\ccr} of 
	  $I(E_4), I(E_3)$ which does not create cut-obstacles, the directed edge $e'$ does not belong to the correct path 
	  that contains $(z,b)$. Therefore, without loss of generality,
	  we can assume that the configuration locally depicted in Figure~\ref{fig.case1ii} is obtained
	  by the {\ccr} of $I(E_4), I(E_3)$;
	  up to some different location of $\{u_1,x_1,x'\}$.
	  On the obtained mixed graph there exists a correct reduction of $I(E_2)$ that does not create a cut-obstacle at $I(E_1)$.
	  If so, the correct reduction of $I(E_2)$ is so that
	  the arc $(w,a')$ belongs to the correct path that contains $(b',b)$.
	  Because of the fact that $(b,w)$, $(w,a')$ is a forbidden pair of arcs, the correct path $(b',b, w, a')$ exists
	  and does not contain the arc $(b,w)$. 
	  Moreover, since the pair $(b,w)$, $(u_1,b)$ is also forbidden, there exists the correct path $(u_1,b,b')$.   
	  But then, it holds that the arc $(b,w)$ and the remaining 
	  arc $(w,b)$ (arising from the edge $\{w,b\}$) belong to the same correct path or cycle. 
	  However, it does not correspond to a correct reduction since $(a,w)$ and $(w,a)$ are forced to be in the same cycle.
	  \end{proof}

It is a routine to check that the following list of reductions (encoded by their correct paths and cycles) 
correspond to all possible 4 distinct correct reductions of~$I(E_4)$.

\vspace*{0.2cm}
\begin{tabular}{ll}
\vspace*{0.2cm}
Reduction 1: & $(y,v,w)$, $(a',v,u,z)$, $(z,u,u_2)$ and $(u_1,u,v,y)$. See Figure~\ref{fig.case1}.\\ \vspace*{0.2cm}
Reduction 2: & $(z,u,v,w)$, $(a',v,y)$, $(y,v,u,u_2)$ and $(u_1, u, z)$. See Figure~\ref{fig.case2}.\\\vspace*{0.2cm}
Reduction 3: & $(z,u,v,y)$, $(y,v,w)$, $(u_1,u,z)$ and $(a',v,u,u_2)$. See Figure~\ref{fig.case3}.\\\vspace*{0.2cm}
Reduction 4: & $ (y,v,u,z), (a',v,y), (z,u,u_2)$ and $(u_1,u,v,w)$. See Figure~\ref{fig.case4}\vspace*{0.2cm}
\end{tabular}

\begin{figure}[h]
\centering
\subfigure[Case 1]
{
\ifpdf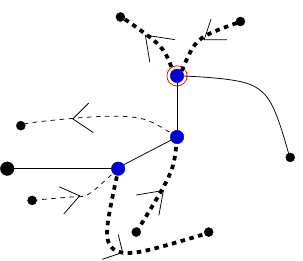\else\input{case1.ps_t}\fi
\label{fig.case1}
}\qquad \qquad
\subfigure[Case 2]
{
\ifpdf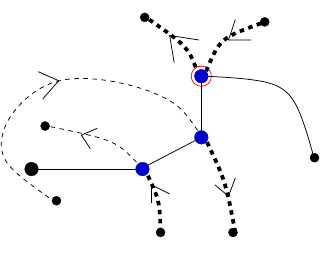\else\input{case2.ps_t}\fi
 \label{fig.case2}}
 
 \vspace*{0.2cm}
 
\subfigure[Case 3]
{
\ifpdf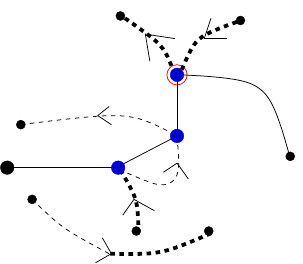\else\input{case3.ps_t}\fi
 \label{fig.case3}}\qquad \qquad
 \subfigure[Case 4]
{
\ifpdf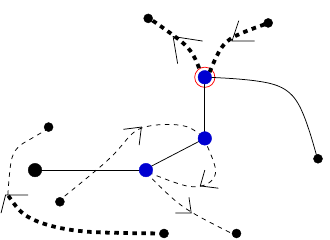\else\input{case4.ps_t}\fi
 \label{fig.case4}}
 \caption{Resulting configurations of all possible correct reductions of $I(E_4)$. 
 The arc incident to $w$ contains $e$ and the arc incident to $a'$ contains $e'$.}
\label{fig.reductionE4}
\end{figure}

Recall that we aim to prove that either $\alpha_1=\beta_1$, or $\alpha_2=\beta_2$. 

We first study reduction 1 (the analysis of reduction 2 follows in an analogous way). 
Since we are examining the steps of a complete reduction process of $\mathcal{G}$ on $H'$,
we have that $I(E_3)$ is not a cut-obstacle; thus, there exists $i\in\{1,2\}$ such that $\alpha_{i} = \beta_i$.
Hence the study of reduction 1 involves the following three cases: (i) $\alpha_{1} = \beta_1$ and 
$\alpha_{2} \neq \beta_2$, (ii) $\alpha_{1} \neq \beta_1$ and $\alpha_{2} = \beta_2$,
and (iii) $\alpha_{1} = \beta_2$ and $\alpha_{1} = \beta_2$. 

By Observation~\ref{obs:useful}, if case~(i) holds, then a correct reduction of $I(E_3)$  
takes the correct path $(b,z,y,w)$ and makes that $e'$ belongs to the correct path that contains $(z,b)$.
Therefore, in a correct reduction of $I(E_3)$ which does not create cut-obstacles the following
holds: the arc $(z,u_2)$ ($\alpha_2$, respectively) is in the correct path that contains $(x',x)$ ($(x,x')$, respectively);
thus the statement of Theorem~\ref{theo:important} follows.

\begin{figure}[h]	
\centering
 \subfigure[A correct reduction of $V(E_3)$ in Case 1(i).]
{
\ifpdf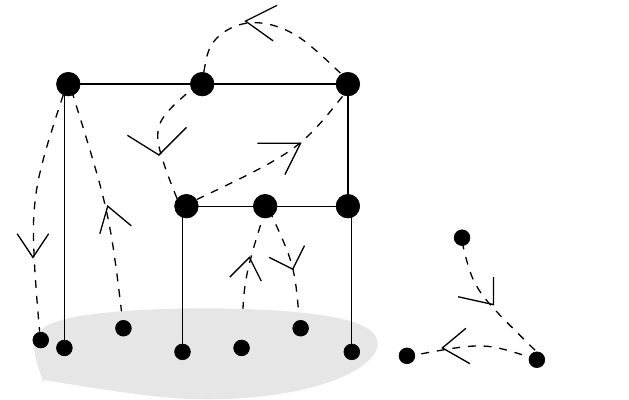\else\input{case1i.ps_t}\fi
\label{fig.case1i}}\qquad
 \subfigure[A correct reduction of $V(E_3)$ in Case 1(ii).]
{
\ifpdf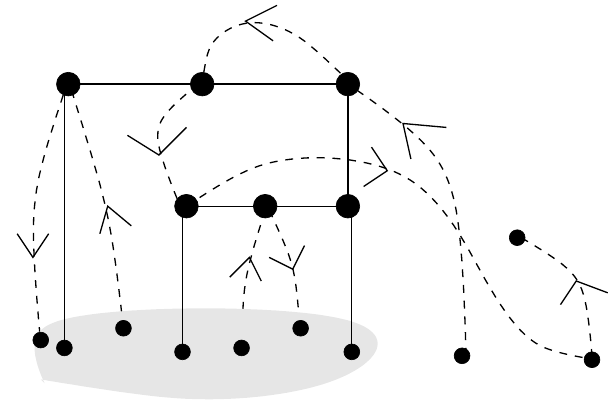\else\input{case1ii.ps_t}\fi
\label{fig.case1ii}}
\caption{}
\end{figure}

Now we study case~(ii). Recall that reduction~1 is depicted in Figure~\ref{fig.case1};  to obtain case (ii) we need to set $(z,x)=(z,u_2)=(x_2,x)$.
Let us suppose that there exists a correct reduction of $I(E_3)$ that does not create cut-obstacles.
By Observation~\ref{obs:useful}, this implies that $e$ and $(b,z)$ belong to the same correct path.
Since the pair $e$, $(z,x)$ is not a forbidden pair of arcs, then there are two 
potential correct paths such that $e$ and $(b,z)$ could belong to (in case that the required correct reduction exists); 
namely $e$ belongs to either 
$(b,z,y,w)$, or to $(b,z,x,y,w)$. 
If the correct reduction takes the correct path $(b,z,y,w)$, then by Observation~\ref{obs:useful}
this correct reduction also takes the correct path $(a',z,b)$. 
We complete the proof using the same argument as for case (i); meaning,
if a correct reduction of $I(E_3)$ exists then $(u_1, y)$ ($\alpha_1$, respectively)
belongs to the correct path that contains $(x,x')$ ($(x',x)$, respectively).
We now suppose that the correct reduction takes the correct path $(b,z,x,y,w)$.
This forces the correct reduction to consider the following correct paths:
$$(u_1,y,z,b), (a',z,y,x,x'), (x',x,x_1).$$
However, this contradicts Observation~\ref{obs:useful}, since $e'$ and $(z,b)$ must belong to the same correct path.

We now move to study case (iii). We can obtain a local sketch of this case if
we set $(x,y)=\alpha_1=(u_1,y)$ and $(z,x)=\alpha_2 =(z,u_2)$,
in Figure~\ref{fig.case1}.
First of all, by Observation~\ref{obs:useful}, the correct path $(a',z,b)$ is created and 
therefore, the correct path $(b,z,y,w)$ exists as well. Hence,
the correct cycle $(y,x,z,y)$ exists. It implies that the arcs $(x,x')$, $(x',x)$ (arising from the edge $\{x,x'\}$) 
are in the same cycle; a contradiction to the definition of correct reduction.

Let us study reductions~3 and~4. By Observation~\ref{obs:useful},
on the one hand, in the case of reduction 3 we have $(a',u_2)= (x_2,x)$ in Figure~\ref{fig.case3}. 
On the other hand, in the case of reduction~4 we must have $(u_1,w)=(x,x_1)$ in Figure~\ref{fig.case4}.
Therefore, reductions 3 and 4 are analogous and it suffices to study one of them.
We study reduction 3. In any correct reduction of $I(E_3)$ that does not create cut-obstacles, there exists only one 
possible correct path, namely $(b,z,y,w)$, such that Observation~\ref{obs:useful} is satisfied; 
because if there were a different correct path, then this correct path 
would contain $(u_1,z)$ and would 
required that $(u_1,z)=(x,x_1)$, but such a path cannot contain $(b,z)$ since
the direction of $(u_1,z)$ is opposite to the one of $(b,z)$ in any potential
correct path.
Moreover, since the pair $(u_1,z)$, $(z,y)$ is forbidden, then the correct path $(u_1,z,b)$ exists.
Hence, $(a',x)$ is not in the correct path containing $(z,b)$, a contradiction to the statement of Observation~\ref{obs:useful}.

This conclude the first part of the proof of Theorem~\ref{theo:important}.

For the second part of the proof, we consider an $j$-incomplete reduction process of $\mathcal{G}$
for some $2\leq j\leq 4$ and have to show that it implies the existence of a cut-obstacle at $I(E_{j-1})$.
By definition of $j$-incomplete reduction process and Theorem~\ref{thm.main2}, the desired result follows
from the fact that no {\ccr} of $I(D_3), I(D_2),  I(D_1),  I(E_4), \ldots, I(E_j)$ 
on $H'$ creates an inner obstacle at $I(E_{j-1})$ (see Definition~\ref{def.inner}).
A quick examination of Figure~\ref{fig:gadget} show us that if $j\in\{2,3\}$, then
there is no inner obstacle at $I(E_{j-1})$; in the case that $j=2$ (resp. $j=3$),
note that all arcs incident to $w'$ (resp. $a$) 
are not incident to any other vertex of $I(E_1)$ (resp. $I(E_2)$).
In the case that $j=4$, if an inner obstacle at $I(E_3)$ were created,
then it would be required that $\alpha_1= \beta_1$ and $\alpha_2= \beta_2$, because
 in an inner obstacle there are exactly 2 arcs that do not have all its end vertices in 
 $I(E_3)$. Moreover, it would be needed that the arc that connects $x$ and $z$  forms a forbidden pair 
 with the arc that has exactly one-end vertex in $I(E_3)$ and this end vertex is $y$; however, this does never occur. 
\end{proof}

The following theorem states that the condition either $\alpha_1=\beta_1$, or $\alpha_2=\beta_2$
is also sufficient for the existence of a complete reduction process of $\mathcal{G}$.

\begin{theorem}\label{theo:existence}
There exists a complete reduction process of $\mathcal{G}$
if and only if either $\alpha_{1} = \beta_1$ and $\alpha_{2} \neq \beta_2$, 
or   $\alpha_{2} = \beta_2$ and $\alpha_{1} \neq \beta_1$.
\end{theorem}

\begin{proof}
The necessity of the condition is stated in Theorem \ref{theo:important}. 
It remains to show that the condition is sufficient.
Without loss of generality, $\alpha_{1} = \beta_1$ and 
$\alpha_{2} \neq \beta_2$ can be assumed.  
We want to prove that a {\ccr} of $\mathcal{G}^{-1}$ 
that creates no cut-obstacles at $I(E_3)$, $I(E_2)$ and $I(E_1)$ exists.

By the proof of Theorem~\ref{theo:important}, it is possible to correctly reduce $I(D_3), I(D_2), I(D_1), I(E_4)$ in 
such a way that we end up in Case 1 (depicted in Figure~\ref{fig.case1}).
We consider the notation from Figure~\ref{fig.case1}; 
recall that we have $(x,x_1)=(u_1,y)$. 
We reduce $I(E_3)$ in such a way that the following list of correct paths  and cycles
determines the reduction: 
$$(b,z,y,w), \quad (a',z,b), \quad (x',x,y,z,u_2), \quad (x_2,x,x'), \quad (x,y,x).$$

This reduction is correct and we obtain the configuration depicted in Figure~\ref{fig.case1i}.

\begin{figure}[h]	
\centering
  \subfigure[A correct reduction of $I(E_3)$ in Case 1(i).]
{\ifpdf\input{case1i.pdf_t}\else\input{case1i.ps_t}\fi
\label{fig.case1i}}\qquad
  \subfigure[A correct reduction of $I(E_3)$ in Case 1(ii).]
{
\ifpdf\input{case1ii.pdf_t}\else\input{case1ii.ps_t}\fi
\label{fig.case1ii}}
\caption{}
\end{figure}

After this reduction of $I(E_3)$, we complete the {\ccr} by reducing $I(E_2)$ and $I(E_1)$, respectively
by the reductions determined by the lists of correct paths  and cycles:
$$ (b,w,b), (a',b,b'), (b',b, w, a, a_1), (a'', a, w, a'), (a_2, a, a'') \,\,\,\, \text{and}$$
$$ (a'',a',\tilde{a}), (\tilde{a}, a', w', w'_1), (w'_2, w', b', \tilde{b}), (\tilde{b}, b', a_1), (a',b',w',a'), \,\,\, \text{respectively.}$$
\end{proof}

\subsection{Gadgets*, gadgets** and double gadget}
\label{sub.gg}

Note that, in the construction of the canonical ear decomposition of $H(G)$, 
there are two blocks of ears involved; both of these blocks are extensions of basic gadgets.
One of them is block~(\ref{eq:gadget}), and it corresponds
to the gadget, which is defined in Definition~\ref{def:gadget} and whose 
analysis is worked out earlier in Section~\ref{sec:main1}.
In addition to the gadget, a slightly different block of ears 
arises (namely, block~(\ref{eq:gadget*})); 
we shall refer to this block as a \emph{1-gadget*}. It turns out that, 
gadgets and 1-gadgets* behave exactly in the same way with respect to 
consecutive correct reductions without cut-obstacles.

 In addition, we introduce other useful extensions of basic gadgets, namely a \emph{2-gadget*,
1-gadget**, 2-gadget**}.

\begin{figure}[h]
  \subfigure[Gadget* after reduction of its stars.]
{
\ifpdf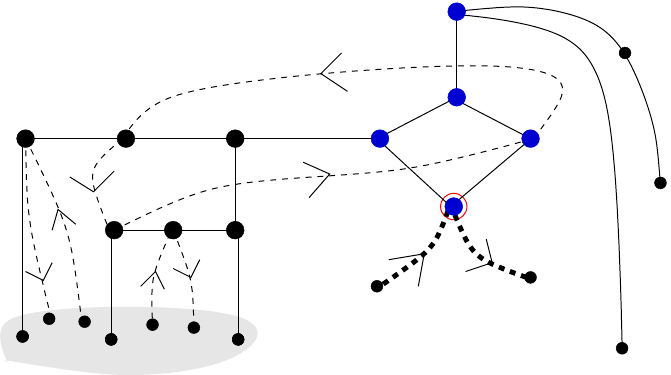\else\input{gadget1.ps_t}\fi
 \label{fig:gadget*}}\qquad
  \subfigure[Gadget** after reduction of its stars.]
{
\ifpdf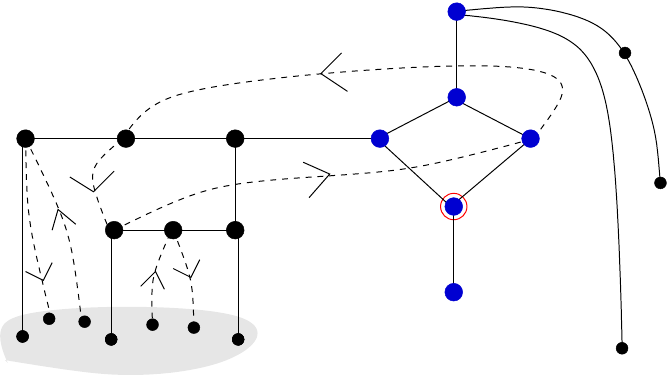\else\input{gadget2.ps_t}\fi
 \label{fig:gadget**}}\caption{A 1-gadget* (resp. 1-gadget**) contains the 1-ear on vertex set $\{x',x,s\}$,
 and a 2-gadget* (resp. 2-gadget**) contains the 2-ear on vertex set $\{x',x,s,x^*\}$.
 In the case of a 2-gadget* or a 2-gadget**, we say that vertex $s$ is the \emph{pending vertex} of the gadget.}\label{fig:*and**}
 \end{figure}

Let $(H_0, H_i, L_i)^n$  be an ear decomposition of a trigraph $H$
and $\mathcal{B}=\mathcal{B}(x,u)$ be a basic gadget of $H$.
A \emph{1-gadget*} $\mathcal{G}$ of $(H_0, H_i, L_i)^n$ associated
to $\mathcal{B}$ is the block of ears
\begin{equation}
 \label{eq:1gadget*} \mathcal{G}= F, \mathcal{B}\end{equation}
                                                 of $L_1, \ldots, L_n$, where $F$ is the 1-ear with internal vertex $x$ and end vertices not in $V(\mathcal{B})$;
in the case that $F$ is a 2-ear with internal vertices $x, s$ and $s$ is neither a joint vertex, nor a replica vertex
of a basic gadget of $H$, we say that $\mathcal{G}$ is a \emph{2-gadget*}. We refer to 1-gadgets* and to 2-gadgets* 
as \emph{gadgets*}.

In addition to gadgets*, we need to introduce a natural extension of them; namely, \emph{gadgets**}.
Informally, a gadget** is obtained from a gadget* by removing the 2-ear $E_4$ and adding
two new 1-ears instead.
A gadget** $\mathcal{G}$ of $(H_0, H_i, L_i)^n$ associated to $\mathcal{B}$ is the block of ears
\begin{equation}
 \label{eq:1gadget**}
\mathcal{G}= F, E_1(\mathcal{B}), E_2(\mathcal{B}), E_3(\mathcal{B}), F_1, F_2, D_1(\mathcal{B}), D_2(\mathcal{B}), D_3(\mathcal{B}),
\end{equation}
where $F$ is as described for gadgets* (accordingly we have 1-gadgets** and 2-gadgets** --- see Figure~\ref{fig:*and**}),
$F_2$ is the 1-ear with vertex set $\{y,v,u\}$ and $F_1$ is the 1-ear with vertex set $\{z,u,u'\}$
with $u'$ the neighbour of $u$ that is not in $V(\mathcal{B})$ (as described in Figure~\ref{fig:gadget**}).

Again, the terminology defined for basic gadgets and gadgets is naturally transfered to gadgets* and gadgets**.
Further, we suppose that  $D_3(\mathcal{B})=L_m$ and that Assumption~\ref{assumption1} holds.
If $\mathcal{G}$ is a gadget*, suppose that $\beta_1$ and $\beta_2$ are the arcs incident 
to the replica vertex of $\mathcal{G}$ with the orientations
according to Figure~\ref{fig:gadget*}. 
If $\mathcal{G}$ is a gadget**, let $\beta_1$ denote the arc $(u',u)$
and $\beta_2$ the arc $(u,u')$.
The following statement for the gadgets* and gadgets** follows from the proof of Theorems~\ref{theo:important} and~\ref{theo:existence}.

\begin{theorem}\label{theo:bstargad}
Let $\mathcal{G}$ be a gadget* or a gadget**.
If there exists a {\ccr} of $\mathcal{G}^{-1}$
that creates no cut-obstacles at a 3-ear from $\mathcal{G}$, 
then $\beta_1$ and $\beta_2$ belong to the two distinct correct paths that contain $(y,x)$ and $(x,y)$,
respectively. Moreover, such a {\ccr} always exists.
\end{theorem}

Because of technical reasons, we present some extra block of ears, which are concatenations of basic gadgets
in gadget* and/or gadget** fashion. 
We refer to them as \emph{double gadgets} and illustrate them in Figure~\ref{fig:double}.
A block of ears $\mathcal{D}$ of $L_{1}, \ldots, L_{n}$  is called a \emph{double gadget} if 
\begin{equation*}
\mathcal{D} = F', \mathcal{G}_{\mathcal{B}}-F(\mathcal{G}_{\mathcal{B}}), \mathcal{G}_{\mathcal{B}'}-F(\mathcal{G}_{\mathcal{B'}})  
\end{equation*}
where $\mathcal{B}=\mathcal{B}(x,u)$, $\mathcal{B}'=\mathcal{B}'(x',u')$ are basic gadgets, 
$\mathcal{G}_{\mathcal{B}}-F(\mathcal{G}_{\mathcal{B}})$ and $\mathcal{G}_{\mathcal{B}'}-F(\mathcal{G}_{\mathcal{B'}})$
are obtained from gadgets* or gadgets** $\mathcal{G}_{\mathcal{B}}$ and $\mathcal{G}_{\mathcal{B}'}$ 
associated to $\mathcal{B}$ and $\mathcal{B}'$ respectively (see~(\ref{eq:1gadget*}) and~(\ref{eq:1gadget**}))
by removing the ear $F(\mathcal{G}_{\mathcal{B}})$ and $F(\mathcal{G}_{\mathcal{B'}})$, respectively, and  
$F'$ is a 2-path with internal vertices $x$, $x'$ and end vertices
not in $V(\mathcal{G}_{\mathcal{B}} \cup \mathcal{G}_{\mathcal{B}'})$.
We denote a double gadget by $\mathcal{D}(\mathcal{B},\mathcal{B}')$.

\begin{figure}[h]	
\centering
\ifpdf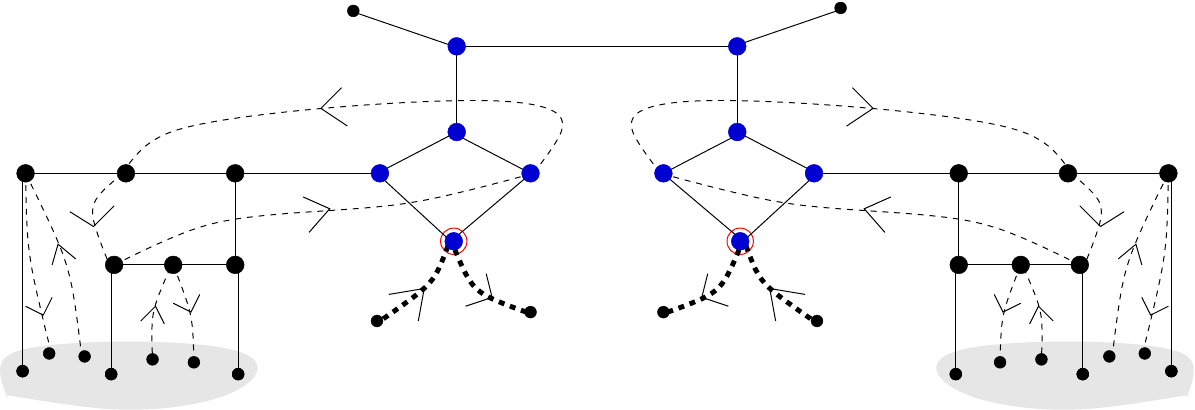\else\input{double.ps_t}\fi
\caption{Example of a double gadget where $\mathcal{G}_{\mathcal{B}}-F(\mathcal{G}_{\mathcal{B}}), \mathcal{G}_{\mathcal{B}'}-F(\mathcal{G}_{\mathcal{B'}})$ are obtained from gadgets*.}
\label{fig:double}
\end{figure}

As expected double gadgets have the same behavior, with respect to
the reduction process, as gadgets* and gadgets** do. Therefore, its analysis follows from the discussion regarding gadgets.

As before, we suppose that the last ear of the double gadget $\mathcal{D}$ is  $L_m$ and that Assumption~\ref{assumption1} holds. 
Note that in $\mathcal{D}(\mathcal{B},\mathcal{B}')$, 
when $\mathcal{G}_{\mathcal{B}}$ and $\mathcal{G}_{\mathcal{B}'}$ are gadgets*, 
the replica vertices of $\mathcal{B}$ and $\mathcal{B}'$ have degree 2 in the mixed graph obtained by the {\ccr}
of $I(L_n),\ldots I(L_{m+1})$ and therefore each of them is incident to two arcs,
say $\beta_1$, $\beta_2$ in the case of the replica vertex of $\mathcal{B}$,
and $\beta'_1$, $\beta'_2$ in the case of the replica vertex of $\mathcal{B}'$ (their directions
are as in Figure~\ref{fig:double}). In the case that $\mathcal{G}_{\mathcal{B}}$ is
a gadget**, we have $u$ has degree 3. If $u''$ denotes the neighbour of $u$ that does not belong to $V(\mathcal{B})$,
we denote $\beta_1=(u'',u)$ and $\beta_2=(u,u'')$; analogously for
$\mathcal{G}_{\mathcal{B}'}$. We point out that possibly $\beta_1= \beta_2'$ or $\beta_2= \beta_1'$.

The following statement for double gadgets follows.

\begin{theorem}\label{theo:double}
Let $\mathcal{D}=\mathcal{D}(\mathcal{B},\mathcal{B}')$ be a double gadget 
such that $\{b,z,y\}$ (resp. $\{b',z',y'\}$) is the {\emph{fixed}} path,
$u$ (resp. $u'$) is the \emph{replica} vertex and $x$ (resp. $x'$) is the \emph{joint} vertex of $\mathcal{B}$ (resp. $\mathcal{B}'$).
If there exists a {\ccr} of $\mathcal{D}^{-1}$ that creates no cut-obstacles at a 3-ear from $\mathcal{D}$,
then the arcs $\beta_1, \beta_2$ (resp. $\beta'_1, \beta'_2$) 
belong to the correct paths that contains $(y,x)$ and $(x,y)$, respectively (resp. $(y',x')$ and $(x',y')$).
Moreover, such a {\ccr} always exists.
\end{theorem}

\subsection{A directed cycle double cover for $G$}

In this section we complete the proof of Theorem~\ref{th:mainmain}; 
we want to prove that Conjecture~\ref{conj:weakavoidance}
implies the {\dcdc} conjecture in general graphs.
Let $G'$, $H(G)$, $S$ and $\Gamma$ as defined in Section~\ref{sec:main}.
We recall that the canonical ear decomposition of $H(G)$ 
admits $H_0,S$ and it is robust provided that graph $G$ is cyclically 3-edge-connected
and the ear decomposition of $G$ used in the construction of $H(G)$ is super robust. 

In order to understand the aim of the next definition, recall that 
Observation~\ref{o.cl} claims that if $H', \mathcal{H}'$ is an element in the closure
of $H(G)$ and its canonical ear decomposition, then every
basic gadget of $H(G)$ is contained in $H'$.

\begin{definition}[Relevant ear decompositions]
\label{def.rell}
Let $\mathcal H$ be the canonical ear decomposition of $H=H(G)$ and
$H', \mathcal H'$, be a pair in the closure of $H$, $\mathcal H$. 
 Set $\mathcal H'=(H_0, H'_i, L'_i)^n$.
 We say that $\mathcal H'$ is \emph {relevant} if 
 the sequence $L'_1, \ldots, L'_n$ (possibly after 
 some reordering) can be partitioned into blocks of ears, with each block
 being either a star (on 3 or 4 vertices), 
 or a gadget, or a gadget*, or a gadget**, or a double gadget.

\end{definition}

Note that, by definition, the canonical ear decomposition is relevant, since its sequence of building
ears can be partitioned into blocks of ears  with each block
consisting of a star on 4 vertices, or of a gadget, or of a 1-gadget*.

The following observation follows directly from the construction of trigraph $H(G)$, and 
it helps understanding relevant ear decompositions
in the closure and the role of 2-gadgets* and 2-gadgets**.
Recall that, by definition, the pending vertex of 2-gadgets* and 2-gadgets**
is neither a joint vertex, nor a replica vertex.

\begin{observation}
\label{obs.2gad*}
Let $\mathcal H=(H_0, H_i, L_i)^n$ be a canonical ear decomposition of $H=H(G)$ and
$H', \mathcal H'=(H_0, H'_i, L'_i)^n$, be a pair in the closure of $H$, $\mathcal H$
such that $\mathcal H'$ is relevant. 
If $\mathcal{G}$ is a 2-gadget* of $\mathcal H'$ with $x$ and $s$ the internal vertices
of $F(\mathcal{G})$ and $x$ the joint vertex of $\mathcal{G}$ and $s$ the pending
vertex of $\mathcal{G}$  (see definition in Figure~\ref{fig:gadget**}), then $s$ is the internal vertex of a star 
on~4 vertices of $\{L_1,\ldots,L_n\}$.
\end{observation}

We now  extend Definition~\ref{def:redprocess}.

\begin{definition}[Reduction process --- extension]\label{def:redprocess2}
Let $(H_0, H_i, L_i)^n$ be an ear decomposition of~$H(G)$.
A {\em reduction process} of $L_1, \ldots, L_n$ is a {\ccr} 
of $I(L_n), \ldots, I(L_j)$ so that
\begin{itemize}
 \item[(i)] either $j= 1$, or
 \item[(ii)] $j>1$ and $I(L_{j-1})$ does not have correct reduction or $I(L_{j-1})$ is a cut-obstacle. 
\end{itemize}
If (i) holds, we say that the reduction process is \emph{complete}.
Otherwise, we refer to it as \emph{$j$-incomplete}.
\end{definition}

In what follows, $\mathcal{H}=(H_0, H_i, L_i)^n$ denotes the canonical ear decomposition of $H=H(G)$. 

Lemma~\ref{lemma:m11} is implied by Theorem~\ref{theo:important}, Theorem~\ref{theo:bstargad} and Theorem~\ref{theo:double}.

\begin{lemma}
\label{lemma:m11}
Let $(H_0, H'_i, L'_i)^n$ be a relevant ear decomposition in the closure of $H$, $\mathcal{H}$. 
Each reduction process of $L'_1, \ldots, L'_n$ satisfies exactly one of the following two statements.
\begin{itemize}
\item[\textendash]
The reduction process is complete.
\item[\textendash]
The reduction process is $j$-incomplete and $I(L'_{j-1})$ is a cut-obstacle. 
\end{itemize}
\end{lemma}

In Lemma~\ref{lemma:main} we state that superb relevant ear decompositions 
encode directed cycles double covers.

\begin{lemma}\label{lemma:main}
Let $H', \mathcal{H}'=(H_0, H'_i, L'_i)^n$ be in the closure of $H$, $\mathcal{H}$
such that $\mathcal{H}'$ is a relevant ear decomposition. 
If $\mathcal{H}'$ is superb, then it encodes a {\dcdc} of $G$.
\end{lemma}

\begin{proof}
By Lemma~\ref{lemma:m11} and definition of superb, 
it follows that a {\ccr} of $I(L'_n), \ldots, I(L'_1)$ that does not create 
cut-obstacle at any 3-ear exists. We show that such a {\ccr} encodes a {\dcdc} of~$G$.

We recall that, by Observation~\ref{o.cl}, the cubic graph $G'$ can be obtained from the trigraph $H'$ by contracting each basic 
gadget of $H$ into a single vertex. 
Since $\mathcal H'$ is relevant, its sequence of ears
$L'_1, \ldots, L'_n$ can be partitioned into blocks of ears 
$\mathcal E_1,\ldots, \mathcal E_t$ where each 
$\mathcal E_i$ is either a star on 3 or 4 vertices, or a gadget, or a gadget*, or a gadget**, or a double gadget.
This partition gives rise to a connected partition $W_1, \ldots W_t$ 
of the union of the internal vertices of the ears of $G'$, 
where each $W_i$ contains either
a single vertex, or two vertices in the following way:
\begin{itemize}
 \item if $\mathcal E_i$ is a star with center $\Gamma(s)$, then $W_i=\{s\}$,
\item if $\mathcal E_i$ is a gadget or a 1-gadget*, or a 1-gadget** with replica vertex $\Gamma(b)$, then $W_i=\{b\}$,
\item if $\mathcal E_i$ is a 2-gadget*, or a 2-gadget** with replica vertex $\Gamma(b)$ and pending vertex $\Gamma(s)$,
then $W_i=\{b,s\}$,
\item if $\mathcal E_i$ is a double gadget with replica vertices $\Gamma(b), \Gamma(b')$, then $W_i=\{b, b'\}$.
\end{itemize}
Recall that function $\Gamma$ (see Definition \ref{def.ggg}) embeds the set of vertices and edges of $G'$ into a subset of the vertices and edges of $H'$. 
Assume now that it is possible to perform a {\ccr} of $W_t, \ldots W_1$. 
After performing this {\ccr} of $W_t, \ldots W_1$ on $G'$, we obtain a mixed graph
on the vertex set $V_0$, where the underlying graph is a triangle.
Moreover, by Observation~3 of~\cite{ourwork}, we have that each mixed triangle has a correct reduction; 
thus, by Proposition~\ref{prop:consecorr}, $G'$ has a {\dcdc}, and also $G$ does. 
Thus, to conclude the proof of the lemma, it suffices to prove the following claim.

\begin{claim}\label{cl:cons}
A {\ccr} of $\mathcal E_t^{-1},\ldots, \mathcal E_1^{-1}$
induces a {\ccr} of~$W_t, \ldots W_1$.  
\end{claim}

The rest of the proof is devoted to prove Claim~\ref{cl:cons}.
Let $i \in \{1, \ldots, t\}$. We assume that a {\ccr} of $\mathcal E_t^{-1},\ldots, \mathcal E_{i}^{-1}$ 
and the corresponding {\ccr} of $W_t, \ldots W_{i}$
is performed on $H'$ and $G'$, respectively. Let $M(i)$ and $M(G',i)$ 
denote the generated mixed graphs, with $M(t)=H'$, $M(G',t) =G'$. 
Without loss of generality, $\mathcal E_t^{-1}$ is a star, and hence
the correct reduction of $W_t$ is natural since $W_t$ is the center of a star.
Thus, it suffices to prove that the correct reduction of $\mathcal E_{i-1}^{-1}$ on $M(i)$
induces a correct reduction of $W_{i-1}$ on $M(G',i)$.

The following statements, namely Properties~\ref{prop:keep1} and~\ref{prop:keep2}, 
are satisfied for every $u, v \in V(G')$ and for $j = t$. 
We assume that they also hold for all $j\geq i$ and we argue that they are satisfied for~$i-1$.

\begin{property}\label{prop:keep1}
\emph{The arc $(u,v)$ is in $M(G',j)$ if and only if there exists an arc $\vec{a}$ 
in $M(j)$ with its head in $\{x_{\Gamma(v)}, \Gamma(v)\}$ and its
tail in $\{x_{\Gamma(u)}, \Gamma(u)\}$; in other words, after
vertex-contraction of all basic gadgets of $M(i)$ into their replica vertices,
the arc  $(\Gamma(u), \Gamma(v))$ exists. Recall that $x_{\Gamma(u)}$ denotes 
the joint vertex of the basic gadget with replica $\Gamma(u)$; in case $\Gamma(u)$
is not part of a basic gadget, set $x_{\Gamma(u)}=\Gamma(u)$.
By abuse of notation, we denote $\vec{a}$ by $\Gamma((u,v))$.
Furthermore, note that if a pair of arcs is forbidden in $M(G',i)$, 
then their images under $\Gamma$ form a pair of arcs forbidden in $M(i)$.}
\end{property}

\begin{property}\label{prop:keep2}
\emph{Suppose that $(u',v')$ is an arc that belong to the correct path of 
the correct reduction of $W_i$ on $M(G',i+1)$ which is replaced by arc $(u,v)$ in order to obtain $M(G',i)$.
Then, either $(u',v')$ is an arc of $M(G',i+1)$, or $\{u',v'\}$ is an edge of $M(G',i+1)$.
If  $(u',v')$ is an arc of $M(G',i+1)$, then the arc $\Gamma((u',v'))$ of $M(H(G),i+1)$ (defined in Property~\ref{prop:keep1})
is in the correct path replaced by the arc $\Gamma((u,v))$. 
If $(u',v')$ is an arc obtained by the orientation of edge $\{u',v'\}$ of $M(G',i)$, then the corresponding orientation of 
$\Gamma(\{u',v'\})$ is in the correct path replaced by the arc $(\Gamma((u,v))$.}
 \end{property}

Assuming validity of both properties for $i$, it is a routine to check that it is also valid for $i-1$ by
using Theorems~\ref{theo:important},~\ref{theo:existence},~\ref{theo:bstargad} and~\ref{theo:double}.
 This finishes the proof of the statement of Claim~\ref{cl:cons}, and also of Lemma~\ref{lemma:main}.
\end{proof}

Finally, Theorem~\ref{th:mainmain} is a straightforward consequence of Lemma~\ref{lemma:main} 
and Lemma~\ref{l.last}. 
We first make a crucial observation, and then formulate Lemma~\ref{l.last}.
As before, $\mathcal H= (H_0, H_i, L_i)^n$ denotes the canonical ear decomposition of $H(G)$.

\begin{observation}\label{obs:suprob}
Let $H', \mathcal H'= (H_0, H'_i, L'_i)^n$ be in the closure 
of $H(G)$, $\mathcal H$. If $\mathcal H'$ is superb, then
$\mathcal H'$ is a robust ear decomposition.
\end{observation}
\begin{proof}
Assume $\mathcal H'$ is not robust.
First of all note that the descendant of $E_1(\mathcal{B}), E_2(\mathcal{B})$, for all basic gadgets $\mathcal{B}$ of 
$H'$ consist of exactly 2 connected components and one of them is an isolated vertex
adjacent to two vertices in $V(H_0)$, since it is exactly the center of the stars $D_2(\mathcal{B}), D_3(\mathcal{B})$, respectively. 
On the other hand, note that the descendant of every 3-ear $E_3(\mathcal{B})$ has at most 2 connected components, and 
there is no component corresponding to an isolated vertex adjacent to vertices in $V(H_0)$.
Thus, by the assumption, for some gadget $\mathcal{G}=\mathcal{G}(\mathcal{B})$ of $H'$, there exists a 3-ear 
$E_3(\mathcal{B})=L'_{m} \in \{L'_1, \ldots, L'_n\}$ such that its descendant has exactly 2 connected components
and hence, there is no path in the descendant of $E_3(\mathcal{B})$ connecting $u$ to $x$; we consider
name of the vertices as in Figure~\ref{fig:gadget}. Note that for every {\ccr}
of $I(L'_n), \ldots, I(L'_{m+1})$ we have $\alpha_1 \neq \beta_1$ and $\alpha_2 \neq \beta_2$.
Therefore, by Theorem~\ref{theo:important}, each reduction process of $\mathcal{G}$ is 
$j$-incomplete and $I(E_{j-1}(\mathcal{G}))$ is a cut-obstacle for some $2\leq j \leq 4$.
This implies that $\mathcal H'$ is not superb, a contraction.
\end{proof}

\begin{lemma} \label{l.last}
Let $H', \mathcal H'= (H_0, H'_i, L'_i)^n$ be in the closure 
of $H(G)$, $\mathcal H$. Assume there exists
a {\ccr}, say $\mathcal R$, of $I(L'_n), \ldots, I(L'_1)$ that makes $\mathcal H'$ superb. 
Then there is a pair $H'', \mathcal H''= (H_0, H''_i, L''_i)^n$ in the closure of $H(G)$, $\mathcal H$ such that
$\mathcal H''$ is a relevant ear decomposition and $\mathcal R$ 
induces a {\ccr} $\mathcal R'$ of $I(L''_n), \ldots, I(L''_1)$ that makes $\mathcal H''$ superb.
\end{lemma}

\begin{proof}
Note that by Observation~\ref{obs:suprob}, we have $\mathcal H'$ is a robust ear decomposition.
In the following, we show how to obtain $H'', \mathcal H''= (H_0, H''_i, L''_i)^n$
of Lemma~\ref{l.last} from $H', \mathcal H'= (H_0, H'_i, L'_i)^n$
by modifying $\mathcal H'$ only; meaning $H'=H''$. 
%
%
%
%

Since $H', \mathcal H'$ is in the closure of $H(G)$, $\mathcal H$, we get that
 $\mathcal H'$ admits $H_0, S$ (see Definition~\ref{def.S}).
By Observation~\ref{o.cl},
for each basic gadget $\mathcal{B}$ of $H(G)$, the ears 
\begin{equation*}\label{eq:core} E_1(\mathcal{B}),E_2(\mathcal{B}),
E_3^*(\mathcal{B}), D_1(\mathcal{B}),D_2(\mathcal{B}), D_3(\mathcal{B})\end{equation*} belong to $\{L'_1,\ldots,L'_n\}$,
where $E_3^*(\mathcal{B})$ is a 2-ear or a 3-ear that contains the fixed path of $\mathcal{B}$,
which is an element of $S$ by construction; recall that 
(as it is mentioned after the construction of $H(G)$), the ear $E_1(\mathcal{B})$
is a base, $E_2(\mathcal{B})$ is an up, $E_3^*(\mathcal{B})$ is an antenna,
and by condition (3) in the second part of Definition~\ref{def.S}, the
stars $D_1(\mathcal{B}), D_2(\mathcal{B}), D_3(\mathcal{B})$ cannot be modified.
In this proof, for a given $\mathcal B$, we call the sequence of ears $E_1(\mathcal{B}),E_2(\mathcal{B}),
E_3^*(\mathcal{B})$
the \emph{core} of $\mathcal B$. Without loss of generality, from now on, we assume that 
the list of ears $L'_1, \ldots, L'_n$ is ordered so that
each core forms a block. Therefore, we can consider 
a natural ordering, say $\prec$, of the basic gadgets of $H(G)$: 
$\mathcal{B} \prec \mathcal{B}'$ if $E_3(\mathcal{B})=L_{k}$, $E_3(\mathcal{B}')=L_{m}$ and $k<m$.

We denote by $\mathcal L$, the set of ears consisting of the cores of all basic gadgets of $H$.
We have every 3-ear of $\mathcal H'$ is in $\mathcal L$, because: $\mathcal H'$ admits $H_0, S$, which 
implies that if $L$ is a 3-ear of $\{L'_1, \ldots, L'_n\}$, then $L$ is either a base, or an up, or an antenna containing
an element of $S$. Hence, by Remark~\ref{o.cl1},~$L \in \mathcal L$.

From now on, we refer to a gadget, to a gadget*, to a gadget** and to a double gadget
as a \emph{block-gadget}.

If  every basic gadget of $H(G)$ is contained is some block-gadget of $\mathcal H'$,
then we have that $\mathcal H'$ is relevant and we can put $\mathcal H''= \mathcal H'$.
Therefore, we assume that the set of basic gadgets of  $H(G)$  that are not contained in some block-gadget
of $\mathcal H'$ is not empty and we denote the set of such basic gadgets of $H(G)$ by $\Sigma(\mathcal H')$.

The following claim therefore proves Lemma~\ref{l.last}.

\begin{claim}\label{cl.last}
There exists an ear decomposition $\mathcal H^*$ of $H'$  
such that $|\Sigma(\mathcal H^*)| < |\Sigma(\mathcal H')|$ and $R$ induces a {\ccr}, say $\mathcal R^*$, of  $\mathcal H^*$. Moreover, each 3-ear in the set of building ears of $\mathcal H^*$ also belongs to $\{L'_1, \ldots, L'_n\}$.
Further, $\mathcal H^*$ is superb and
in the closure of $H(G)$, $\mathcal H$.  
\end{claim}

The rest of this proof is devoted to prove Claim~\ref{cl.last}.

We obtain $\mathcal H^*$ from $\mathcal H'$ by modifying the list $L'_1, \ldots, L'_n$.

Let $\mathcal{B}$ be the first basic gadget, according to the order $\prec$, such that
$\mathcal{B} \in \Sigma(\mathcal H')$ (namely, is not contained in a block-gadget of $\mathcal H'$). 
There are two possible scenarios: (1) either the ear $E_4(\mathcal{B})$ is in $\{L'_1, \ldots, L'_n\}$, (2) or not.
\begin{itemize}
 \item[Case (1)] $E_4(\mathcal{B})$ is in $\{L'_1, \ldots, L'_n\}$.
\end{itemize}
We assume the first ear of $\mathcal{B}$ is $L'_r$.
If $E_4(\mathcal{B})$ is in $\{L'_1, \ldots, L'_n\}$, then $E_3^*$ is a 2-ear, otherwise,
$\mathcal{B}$ would belong to a gadget. Let $x$ be the joint vertex of $\mathcal{B}$.
Clearly, $x$ is a leaf of $E_3^*$. 
This implies that there exists an ear $L \in \{L'_1, \ldots, L'_{r-1}\}$
that contains $x$ as an internal vertex, with $L$ a 1-ear or a 2-ear. 
If $L$ was a 1-ear, then $\mathcal{B}$ would be contained
in a 1-gadget*. 
Thus, $L$ is a 2-ear. Since $L$ is not contained in a 2-gadget*, we have that
for $I(L)=\{x,x'\}$, the vertex $x'$ is the joint or the replica vertex  of a basic gadget $\mathcal{B}'\neq \mathcal{B}$.
If $x'$ was the replica vertex of $\mathcal{B}'$, then $\mathcal{B}'$ would not be contained in a block gadget (because
$E_4(\mathcal{B}')$ would not be in $\{L'_1, \ldots, L'_n\}$), which contradicts the choice of $\mathcal{B}$.

Therefore, $x'$ is a joint vertex of $\mathcal{B}'$. Since $\mathcal{B}$ is not contained in a double gadget,
then neither $E_4(\mathcal{B}')$, nor both $F_1(\mathcal{G}_{\mathcal{B}'})$, $F_2(\mathcal{G}_{\mathcal{B}'})$ 
belong to $\{L'_1, \ldots, L'_n\}$, where $F_1(\mathcal{G}_{\mathcal{B}'})$, $F_2(\mathcal{G}_{\mathcal{B}'})$,
denote the 1-ears of the gadget** $\mathcal{G}_{\mathcal{B}'}$ (see~(\ref{eq:1gadget**})).
Therefore, $F_2(\mathcal{G}_{\mathcal{B}'}), F'_1 \in \{L'_1, \ldots, L'_n\}$, where 
$F'_1$ is a 2-ear that contains the edges of $F_1(\mathcal{G}_{\mathcal{B}'})$ 
and an extra edge $e$ not in $E(\mathcal{G}_{\mathcal{B}'})$. Clearly, $F'_1$
is not contained in a block-gadget.

In this case we let the list of building ears of $\mathcal H^*$ be obtained by 
removing $F'_1, F_2(\mathcal{G}_{\mathcal{B}'})$ from $\{L'_1, \ldots, L'_n\}$
and adding two new ears $E_4(\mathcal{B}')$
and the 1-ear contained in $F'_1$ that contains $e$ and has end vertex
the replica vertex of $\mathcal{B}'$. Hence, $\mathcal{B}$ and $\mathcal{B}'$
belong to a double gadget of $\mathcal H^*$.
As $F'_1$ is not contained
in a block-gadget of $\mathcal H'$, we have that  $\Sigma(\mathcal H^*)+2=\Sigma(\mathcal H')$. The result follows.

\begin{itemize}
 \item[Case (2)] $E_4(\mathcal{B})$ is not in $\{L'_1, \ldots, L'_n\}$.
\end{itemize}

Let $\{z,u,v,y\}$ be the vertex set of $E_4(\mathcal{B})$, where $u$ is the replica vertex
of $\mathcal{B}$ (as in Figure~\ref{fig:bagadget}). As $E_4(\mathcal{B})$ is not in $\{L'_1, \ldots, L'_n\}$,
the 1-ear, say $F_2$, on vertex set $\{u,v,y\}$ is in $\{L'_1, \ldots, L'_n\}$ 
and a 1-ear or a 2-ear $F_1$ containing $\{z,u\}$ is in $\{L'_1, \ldots, L'_n\}$. 

First let $F_1$ be a 1-ear. If $E^*_3(\mathcal{B})$ is a 2-ear, then the proof
is the same as the proof of Case~(1): just use $F_1(\mathcal{B}), F_2(\mathcal{B})$ instead of $E_4(\mathcal{B})$.
If $E^*_3(\mathcal{B})$ is a 3-ear, then $\mathcal H'$ is not robust: to see this, note that 
the descendant of $E^*_3(\mathcal{B})$ has exactly
2 connected components --- since $v$ is a leaf of the star $D_1(\mathcal{B})$
and $u$ has degree 3 in $F_1 \cup F_2$, there is no path in the descendant
of $E^*_3(\mathcal{B})$ that connects the joint of $\mathcal{B}$ to a vertex from $\{u,v\}$ (see Figure~\ref{fig:gadget}) --- and none of them is an isolated vertex connected to two vertices in $V(H_0)$ (see Definition~\ref{def.wr}). 
Because of Observation~\ref{obs:suprob}, this contradicts the assumption that $\mathcal H'$ is superb.

%
%
%

If $F_1$ is a 2-ear, then  we obtain a list of building ears $\mathcal L$ by
replacing $F_2,F_1$ in $\{L'_1, \ldots, L'_n\}$, by $E_4(\mathcal{B})$ and the 1-ear $F_1 -\{z,u\}$.
Therefore, either $\mathcal{B}$ is in a block-gadget of $\mathcal L$ (which implies $\Sigma(\mathcal H^*)+1=\Sigma(\mathcal H')$), or not. If not, the result follows by Case~(1).

 This finishes the proof of Claim~\ref{cl.last}.

\end{proof}


\section{Ear decompositions of 3-edge-connected cubic graphs}\label{sec:3conn}

Let $G$ be a cubic bridgeless graph. 
Let $G_0, G_1, \ldots, G_l$  and $P_1, \ldots, P_l$ be two sequences of subgraphs of~$G$ 
such that $G_0$ is a cycle of~$G$, $G_l = G$, for each $i \in [l]$,
the subgraph $P_i$ is an ear, $V(P_i) \cap V(G_{i-1})$
is the set of leaves of $P_i$, $E(P_i) \cap E(G_{i-1})=\emptyset$, and 
$G_{i}$ is the union of $G_{i-1}$ and $P_i$, that is, $V(G_{i})= V(G_{i-1}) \cup V(P_i)$ and $E(G_{i})=E(G_{i-1}) \cup E(P_i)$.
We say that $(G_0, G_1, \ldots, G_l, P_1, \ldots, P_l)$, in short $(G_0, G_i, P_i)^l$,
is an \emph{ear decomposition} of $G$. Moreover, in the case that $t \leq l$, we say that 
$(G_0, G_i, P_i)^t$ is a partial ear decomposition of $G$
and if $\{P_1, \ldots, P_l\}$ also contains edges,  then we refer to $(G_0, G_i, P_i)^l$ 
as an \emph{edge+ear decomposition} of $G$.

We now need to generalize the concept of descendant.
Let  $i \in [l]$ such that  $|I(P_i)|\geq 3$ and $V(P_i)= \{\alpha_i, v^1_{i},\ldots, v^{k}_{i},\beta_i\}$, where   $V(P_i) \cap V(G_{i-1})=\{\alpha_i, \beta_i\}$.
We say that $S \subset \{v^1_{i},\ldots, v^{k}_{i}\}$ is a {\em segment} of $P_i$ if 
$|S|\geq 3$ and $S$ induces a connected subgraph of~$P_i$.
Let $G'_i$ denote the graph obtained from $G$ by deleting $V(G_{i})$.
For each segment $S$ of $P_i$, the maximal subgraph $G^S_i$ of $G'_i$
such that $G[V(G^S_i) \cup S]$ is connected is called the \emph{descendant of~$S$}.

\begin{definition}[Super robust ear decomposition]\label{def:superrobust}
An ear decomposition, say $(G_0, G_i, P_i)^l$, of $G$ is \emph{super robust} 
if the following conditions hold:
\begin{itemize}
 \item[(a)]  $G-V(G_0)$ is a connected graph, and
 \item[(b)] for each $i \in [l]$ such that $|I(P_i)|\geq 3$ and for each segment $S$ of $P_i$, 
the descendant of $S$ is connected.
\end{itemize}
\end{definition}

By the construction of the canonical ear decomposition $\mathcal{H}$
of $H(G)$ and the definition of super robust ear decompositions of cubic graphs,
it is immediate that $\mathcal{H}$ is robust if the ear decomposition of $G$ used in the construction of $\mathcal{H}$ is super robust.
Lemma~\ref{lemma:supimplrob} follows.

\begin{lemma} \label{lemma:supimplrob}
The canonical ear decomposition $H(G)$ is robust if the ear decomposition of $G$ chosen for the construction of $H(G)$ is super robust. 
\end{lemma}

It is a well-know fact that the {\dcdc} conjecture holds if and only if 
it holds in the class of 3-edge-connected cubic graphs.
The next lemma (whose proof is in~\cite{Cheriyan}) implies that each 3-edge-connected cubic graph admits a super robust ear decomposition.

\begin{lemma}\label{th:robust}
 Let $G$ be a 3-edge-connected cubic graph.
 There exist $e \in E(G)$ and an ear decomposition $(G_0, G_i, P_i)^l$ of $G-e$ such that
$P_l$ is a path of lenght 2, and for each $i \in \{1, \ldots, l-1\}$, 
 denoting $V_i = V(G_0) \cup V(P_1) \cup \cdots \cup V(P_i)$, the graph $G[V-V_i]$ is connected. 
\end{lemma}

\section{A toroidal example}
\label{sec.surfaces}

Let $G$ be a bridgeless graph embedded on the torus obtained from the toroidal square grid $T$
by deleting a perfect matching composed of horizontal edges only in such a way that
the obtained embedding 
has neither loops in the dual, nor multiple edges.
In addition, we fix a super robust ear decomposition of $G$ starting at cycle $C$
such that $C$ contains only vertical edges of $T$; 
the facial cycles of the embedding are seminal for constructing such ear decomposition.

Let $G'$ be the cubic graph obtained from $G$ as described in Section~\ref{sec:main}.
We take the canonical ear decomposition $\mathcal{H}$ of $H(G')$ and discuss the question:  Does $H(G'), \mathcal{H}$ satisfy Conjecture~\ref{conj:weakavoidance}?

There is a natural directed cycle double cover of $G$, which consists of all the facial cycles 
defined by the embedding of $G$ on the torus. This induces a reduction process, say $\mathcal{R}$, on 
each ear decomposition that belongs to the closure of $H(G'), \mathcal{H}$. In particular, a reduction process on $\mathcal{H}$.

We specify neither the details of the chosen super robust ear decomposition of $G$, nor the canonical ear decomposition of $H(G')$.
However, a critical part of any reduction process occurs at the end; namely, when
sequence of ears of $\mathcal{H}$ corresponding to the vertices of cycle $C$ are reduced. 

In order to study this case, we need to assume that the ears of $\mathcal{H}$ of the part of $H(G')$ 
corresponding to the vertices in $G'\setminus C$ are correctly reduced (according to $\mathcal{R}$) without 
creating cut-obstacles. Let $M$ denote the resulting mixed graph
and $H_M$ denote its underlying graph.
Recall that according to the construction of $H(G')$, graph $H_M$ is obtained following the next steps.
\begin{itemize}
 \item First let $H_0$ be the starting cycle of $H_M$ and $V_0 = \{u, v, w\}$.
 \item Secondly, let $e= \{u', v'\}$ be an edge of $C$ and let $C_0$ the path with set of edges $E(C)-\{e\}$. 
  We attach $C_0$ to $H_0$ by identification of $u$ with $u'$ and $v$ with $v'$.
 \item Finally, for each vertex $z$ of $V(C_0)\setminus \{u',v'\}$ we introduce a basic gadget $B_z := B(x_z,u_z)$ and identify $z$ with $x_z$.
\end{itemize}

We now describe the set of arcs of $M$.
Since the previous correct reductions were performed according to $\mathcal{R}$, 
we have $M$  is obtained from $H_M$ by 
adding two disjoint directed cycles so that the union of their vertex sets is $\{w\} \cup \{u_z: z~\text{ vertex of } V(C_0)\setminus \{u',v'\}\}$. This finishes the construction of $M$.

The canonical ear decomposition of $\mathcal{H}$ induces an ear decomposition $\mathcal{M}$ of $H_M$. 
Let $C_0 = u' z_1 \ldots z_t v'$. The ear decomposition $\mathcal{M}$ consist of gadgets $B_{z_1}, \ldots ,B_{z_{t-1}}$ 
and exactly one 1-gadget* $B_{z_t}$.
We claim that $\mathcal{M}$ can be modified to $\mathcal{M}'$ so that 
$\mathcal{R}$ induces a complete reduction
on $\mathcal{M}'$ without cut obstacles. Indeed, $\mathcal{M}'$ is obtained as follows:
Consider two adjacent vertices $z, z'$ in $\{z_1, \ldots, z_t\}$
such that the directed cycle incident to $u_{z}$ is distinct from the directed cycle incident to $u_{z'}$;
note that these vertices always exist. Then, to obtained the ear decomposition $\mathcal{M}'$,
we replace $\mathcal{B}_{z}$, $\mathcal{B}_{z'}$
by the 2-gadget* $\mathcal{D}(\mathcal{B}_{z},\mathcal{B}_{z'})$.

\newpage

\end{document}